%% file: paper.tex
\documentclass[a4paper, 10pt, journal]{IEEEtran}

\IEEEoverridecommandlockouts                              

\input{preamble}

\title{\LARGE \bf
Saturation-Aware Model Predictive Energy Management \\ for Droop-Controlled Islanded Microgrids*
}

\author{Steffen Hofmann$^{1}$, Ajay Sampathirao$^{1}$, Christian {A.} Hans$^{1}$, Jörg Raisch$^{1}$, Alexander Heidt$^{2}$ and Erich Bosch$^{2}$
\thanks{$^{1}$Technische Universität Berlin, Control Systems Group, Germany, \newline
\texttt{\small\{hofmann, sampathirao,\,hans,\,raisch\}\hspace{0pt}@control.tu-berlin.de}.}%
\thanks{$^{2}$Autarsys GmbH, Berlin, Germany,  \newline
\texttt{\small \{heidt, bosch\}@autarsys.com}.}%
\thanks{*This work was supported by the German Federal Ministry for Economic Affairs and Energy (BMWi), Project {No.} 0324024A.}
}

\begin{document}

\maketitle

\begin{abstract}
In this paper, we propose a minimax \ac{mpc}\hh based \acl{ems} that is robust with respect to uncertainties in renewable infeed and load.
The \ac{mpc} formulation includes a model of low-level droop control with saturation at the power and energy limits of the units.
Robust \ac{mpc}\hh based \aclp{ems} tend to under\hh utilize the renewable energy sources to guarantee safe operation.
In order to mitigate this effect, we further consider droop control of renewable energy sources.

For a microgrid with droop\hh controlled units, we show that enhancing droop feedback with saturation enlarges the space of feasible control actions.
However, the resulting controller requires to solve a mixed\hh integer problem with additional variables and equations representing saturation.
We derive a computationally tractable formulation for this problem.
Furthermore, we investigate the performance gained by using droop with saturation, renewable droop and combination of both in a case study.
\end{abstract}


\input{introduction}
\input{model}
\input{problem}
\input{solution}
\input{caseStudy}

\section{CONCLUSIONS}

In this paper, we have demonstrated that the conservativeness in minimax \ac{mpc} of a \acl{mg} can be reduced by including saturations in droop control and droop control of \acl{res}.
We have derived a tractable reformulation of the optimal control problem over a finite prediction horizon.
This reformulation is based on analytical solutions for the disturbance maximizing the cost function, independently of the control.
In a case study, we have observed that, compared to classical minimax \ac{mpc}, including droop control of \ac{res} provides a benefit,and including both droop control of \ac{res} and saturations in droop control of all units can reduce conservativeness of the minimax \ac{mpc}.

We think evident future work is to extend our model by a detailed battery model and an electrical network.
Also, we would like to further investigate the effects related to power sharing of \ac{res} and extend the reseach to include piecewise\hh affine droop curves or time varying droop parameters.
Furthermore,we plan to extend the robust \ac{mpc} in the following ways.
First, we suppose that our analysis can be extended to more general convex or monotonic costs.
Second,we think the conservativeness of the robust \ac{mpc} can be reduced, for example by formulating a multi\hh objective problem.
Finally, the impact of saturation on computational complexity should be studied.

\appendices
\input{appendix}



\bibliographystyle{IEEEtran}
\bibliography{IEEEabrv,bibliography}

\end{document}

%% file: preamble.tex
\usepackage[T1]{fontenc}
\usepackage[utf8]{inputenc}
\usepackage[ngerman,english]{babel}
\usepackage{csquotes}

\usepackage{amsmath,amssymb,mathtools}
\usepackage{units}
\usepackage{xspace}

\usepackage{xcolor}

\usepackage{acronym}
\acrodef{mm}[MM]{minimax}
\acrodef{mg}[MG]{microgrid}
\acrodef{pv}[PV]{photovoltaic}
\acrodef{ems}[EMS]{energy management system}
\acrodef{mpc}[MPC]{model predictive control}
\acrodef{der}[DER]{distributed energy resource}
\acrodef{res}[RES]{renewable energy systems}
\acrodef{mppt}[MPPT]{maximum power point tracking}
\acrodef{soc}[SOC]{state of charge}
\acrodef{iff}[iff]{if and only if}

\usepackage{booktabs}
\usepackage{tabularx}

\newcommand{\hh}{\nobreak\hspace{0pt}-\hspace{0pt}}

\newcommand{\ie}{i.\,e.\@\xspace}

\newcommand{\eg}{e.\,g.\@\xspace}

\newcommand{\wrt}{w.\,r.\,t.\@\xspace}

\providecommand{\st}{}
\renewcommand{\st}{s.\,t.\@\xspace}
\newcommand{\wolg}{w.\,l.\,o.\,g.\@\xspace}
\newcommand{\Wolg}{W.\,l.\,o.\,g.\@\xspace}

\newcommand{\parenTwoLine}[2]{#1\vphantom{#2}\right.\\ \left.\vphantom{#1}#2}

\newcommand{\T}{^\mathsf{T}}

\newcommand{\Npr}{N_p}
\newcommand{\Nsi}{N_s}

\newcommand{\RR}{\mathbb{R}}
\newcommand{\N}{\mathbb{N}}
\newcommand{\Rp}{\mathbb{R}_{>0}}
\newcommand{\Rne}{\mathbb{R}_{<0}}
\newcommand{\Rpz}{\mathbb{R}_{\geq0}}
\newcommand{\Rnz}{\mathbb{R}_{\leq0}}
\newcommand{\Npz}{\mathbb{N}_{\geq0}}

\newcommand{\tmin}{\text{min}}
\newcommand{\tmax}{\text{max}}

\newcommand{\ttot}{\text{tot}}
\newcommand{\tst}{\text{s}}
\newcommand{\tth}{\text{t}}
\newcommand{\trs}{\text{r}}
\newcommand{\tld}{\text{d}}
\newcommand{\ok}{(k)}

\newcommand{\okk}{(k+1)}
\newcommand{\okm}{(k-1)}
\newcommand{\okj}{(k+j)}

\newcommand{\okjm}{(k+j-1)}
\newcommand{\okNp}{(k+\Npr)}

\newcommand{\odo}{(\cdot)}

\newcommand{\inOneNp}{\in\mathbb{N}_{[1,\Npr]}}

\newcommand{\jInOneNp}{j\inOneNp}

\newcommand{\sca}{^{(1)}}
\newcommand{\scb}{^{(2)}}

\newcommand{\Ts}{\mathrm{T}_\text{s}}

\newcommand{\pk}{p\ok}
\newcommand{\uk}{u\ok}
\newcommand{\pmin}{p^\tmin}
\newcommand{\pmax}{p^\tmax}
\newcommand{\umin}{u^\tmin}
\newcommand{\umax}{u^\tmax}

\newcommand{\wmin}{w^\tmin}
\newcommand{\wmax}{w^\tmax}

\newcommand{\ptox}{\tilde{p}_\ttot}
\newcommand{\pth}{p_\tth}
\newcommand{\pst}{p_\tst}
\newcommand{\prs}{p_\trs}
\newcommand{\prwind}{p_{\trs,1}}
\newcommand{\prpv}{p_{\trs,2}}
\newcommand{\uth}{u_\tth}
\newcommand{\ust}{u_\tst}
\newcommand{\urs}{u_\trs}
\newcommand{\urwind}{u_{\trs,1}}
\newcommand{\urpv}{u_{\trs,2}}
\newcommand{\wrs}{w_\trs}
\newcommand{\wld}{w_\tld}
\newcommand{\pstv}{\tilde{p}_\tst}

\newcommand{\chirwind}{\chi_{\trs,1}}
\newcommand{\chirpv}{\chi_{\trs,2}}

\newcommand{\ptk}{\pth\ok}
\newcommand{\ptkj}{\pth\okj}

\newcommand{\psk}{\pst\ok}
\newcommand{\pskj}{\pst\okj}

\newcommand{\prk}{p_\trs\ok}
\newcommand{\utk}{u_\tth\ok}
\newcommand{\usk}{u_\tst\ok}
\newcommand{\urk}{u_\trs\ok}
\newcommand{\wk}{w\ok}

\newcommand{\wrk}{w_\trs\ok}
\newcommand{\wlk}{w_\tld\ok}
\newcommand{\ptik}{p_{\tth,i}\ok}

\newcommand{\wmink}{\wmin\ok}
\newcommand{\wmaxk}{\wmax\ok}

\newcommand{\ptmin}{p_\tth^\tmin}
\newcommand{\ptmax}{p_\tth^\tmax}
\newcommand{\psmin}{p_\tst^\tmin}
\newcommand{\psmax}{p_\tst^\tmax}
\newcommand{\psminn}{\bar{p}_\tst^\tmin}
\newcommand{\psmaxx}{\bar{p}_\tst^\tmax}
\newcommand{\psmink}{\psminn\ok}
\newcommand{\psmaxk}{\psmaxx\ok}
\newcommand{\prmin}{p_\trs^\tmin}

\newcommand{\prwindmin}{\prwind^\tmin}

\newcommand{\prpvmin}{\prpv^\tmin}

\newcommand{\utmin}{u_\tth^\tmin}
\newcommand{\utmax}{u_\tth^\tmax}
\newcommand{\usmin}{u_\tst^\tmin}
\newcommand{\usmax}{u_\tst^\tmax}

\newcommand{\urwindmin}{\urwind^\tmin}
\newcommand{\urwindmax}{\urwind^\tmax}
\newcommand{\urpvmin}{\urpv^\tmin}
\newcommand{\urpvmax}{\urpv^\tmax}

\newcommand{\pimin}{p_i^\tmin}
\newcommand{\pimax}{p_i^\tmax}

\newcommand{\uimin}{u_i^\tmin}
\newcommand{\uimax}{u_i^\tmax}

\newcommand{\xsmin}{x^\tmin}
\newcommand{\xsmax}{x^\tmax}

\newcommand{\deltath}{\delta_\tth}
\newcommand{\deltati}{\delta_{\tth,i}}
\newcommand{\deltatk}{\deltath\ok}
\newcommand{\deltatkm}{\deltath\okm}
\newcommand{\deltatkj}{\deltath\okj}
\newcommand{\deltatkjm}{\deltath\okjm}
\newcommand{\deltatik}{\deltati\ok}

\newcommand{\pVec}{\mathbf{p}}
\newcommand{\uVec}{\mathbf{u}}
\newcommand{\wVec}{\mathbf{w}}
\newcommand{\xVec}{\mathbf{x}}
\newcommand{\rhoVec}{\boldsymbol{\rho}}
\newcommand{\wminVec}{\wVec^\tmin}
\newcommand{\wmaxVec}{\wVec^\tmax}

\newcommand{\deltathVec}{\boldsymbol{\delta}_\tth}

\newcommand{\rhok}{\rho\ok}
\newcommand{\rholo}{\rho^\tmin}
\newcommand{\rhohi}{\rho^\tmax}
\newcommand{\rhox}{\tilde{\rho}}

\newcommand{\Cth}{C_\mathrm{t}}
\newcommand{\Cst}{C_\mathrm{s}}
\newcommand{\CthOn}{C_{\text{on}}}
\newcommand{\CthSw}{C_{\text{sw}}}

\newcommand{\sat}{\operatorname{sat}}

\newtheorem{remark}{Remark}

\newtheorem{lemma}{Lemma}

\newtheorem{theorem}{Theorem}
\newtheorem{problem}{Problem}

\newcommand{\proof}[1][]{\ifthenelse{\equal{#1}{}}{\noindent\hspace{2em}{\itshape Proof: }}{\noindent\hspace{2em}{\itshape Proof #1: }}\@\xspace}

\usepackage{todonotes}
\newcommand{\PrescientMpc}{Prescient\@\xspace}
\newcommand{\MinimaxMpcResDroopOffDroopSatOff}{MM\@\xspace}
\newcommand{\MinimaxMpcResDroopOffDroopSatOn}{Sat. MM\@\xspace}
\newcommand{\MinimaxMpcResDroopOnDroopSatOff}{RES-droop MM\@\xspace}
\newcommand{\MinimaxMpcResDroopOnDroopSatOn}{Sat. RES-droop MM\@\xspace}

\makeatletter
\renewcommand{\todo}[2][]{\tikzexternaldisable\@todo[#1]{#2}\tikzexternalenable}
\newcommand{\atSh}[1]{\tikzexternaldisable\@todo[inline, color = red!20]{@Steffen: #1}\tikzexternalenable}
\newcommand{\atAs}[1]{\tikzexternaldisable\@todo[inline, color = blue!20]{@Ajay: #1}\tikzexternalenable}
\newcommand{\atCh}[1]{\tikzexternaldisable\@todo[inline, color = green!20]{@Chris: #1}\tikzexternalenable}
\newcommand{\atAll}[1]{\tikzexternaldisable\@todo[inline, color = black!20]{@All: #1}\tikzexternalenable}
\newcommand{\rephrase}[1]{\tikzexternaldisable\@todo[inline, color = red!40]{@doubt: #1}\tikzexternalenable}

\makeatother

\definecolor{vgRed}{RGB}{193, 48, 24}
\definecolor{vgOrange}{RGB}{243, 111, 19}
\definecolor{vgYellow}{RGB}{235, 203, 56}
\definecolor{vgGreen}{RGB}{162, 185, 105}
\definecolor{vgLightBlue}{RGB}{13, 149, 188}
\definecolor{vgDarkBlue}{RGB}{6, 56, 81}

\usepackage{tikz}
\usepackage{pgfplots}
\pgfplotsset{compat=1.12}
\usepgfplotslibrary{fillbetween}

\pgfplotsset{every axis/.append style={semithick,tick style={major tick
			length=4pt,semithick,black}}}

\pgfkeys{/pgfplots/x axis shift down/.style={
		x axis line style={yshift=-#1},
		xtick style={yshift=-#1},
		xticklabel shift={#1}}}

\pgfkeys{/pgfplots/y axis shift left/.style={
		y axis line style={xshift=-#1},
		ytick style={xshift=-#1},
		yticklabel shift={#1},
		scaled y ticks = false,
		y tick label style={/pgf/number format/fixed,
		}
	}
}

\pgfplotsset{myPlot/.style={%
		width=8cm,
		height=2.8cm,
		separate axis lines,
		axis x line*=bottom,
		x axis shift down = 7pt,
		enlarge x limits=false,
		axis y line*=left,
		y axis shift left = 7pt,
		enlarge y limits={abs=.25pt},
		enlarge x limits={abs=.25pt},
	}
}

\pgfplotsset{dataVariationPlot/.style={%
    myPlot,
    width = 4.0cm,
    height = 3.5cm,
    xlabel = {},
    xmin = -1,
    xmax = 1,
    xtick = {-1, 0, 1},
    xticklabels = {worst, nominal, best},
    clip = false,
    legend columns=2,
    legend style={
      draw=none,
      fill=none,
      at = {(1.4, 1.0)},
      yshift = 0mm,
      anchor = south,
      legend cell align=left,
      /tikz/every even column/.append style={column sep=0.3cm}
    },
}
}

\pgfplotsset{prescient/.style = {vgDarkBlue}}
\pgfplotsset{allOff/.style = {vgRed}}
\pgfplotsset{onlySat/.style = {vgOrange}}
\pgfplotsset{onlyResDroop/.style = {vgYellow}}
\pgfplotsset{allOn/.style = {vgGreen}}

\pgfplotsset{prescient dataVariation/.style = {prescient, mark=x, only marks, mark size = 3pt}}
\pgfplotsset{allOff dataVariation/.style = {allOff, mark=o, only marks, mark size = 3pt}}
\pgfplotsset{onlySat dataVariation/.style = {onlySat, mark=square, only marks, mark size = 3pt}}
\pgfplotsset{onlyResDroop dataVariation/.style = {onlyResDroop, mark=triangle, only marks, mark size = 3pt}}
\pgfplotsset{allOn dataVariation/.style = {allOn, mark=diamond, only marks, mark size = 3pt}}

\pgfplotsset{limit/.style = {const plot, color=black, dash pattern=on 1pt off 3pt on 3pt off 3pt, thin}}

\usepgfplotslibrary{external}
\tikzset{external/system call={pdflatex \tikzexternalcheckshellescape -halt-on-error
		-interaction=batchmode -jobname "\image" "\texsource"}}

%% file: introduction.tex

\section{INTRODUCTION}
Decarbonization of the energy sector is promoting a worldwide increase in the deployment of renewable energy~\cite{REN212018}.
This shifts the foundation of power systems from large, centralized generation units to smaller \acfp{der}.
However, the volatility associated with renewable generation has posed serious challenges for their integration into the existing power system setups.
In this context, the \ac{mg} concept represents a promising direction~\cite{ParLotKhoBah2015}.
In~\cite{PES2018}, an \ac{mg} is defined as follows:
``Generally, a microgrid is defined as a group of DERs, including Renewable Energy Resources (RES) and Energy Storage Systems (ESS), and loads, that operate together locally as a single controllable entity.
Microgrids exist in various sizes and configurations;
they can be large and complex networks with various generation resources and storage units serving multiple loads, or small and simple systems supplying a single customer.''

Here, we focus on islanded operation of \acp{mg}.
In such an \ac{mg}, the \acf{ems} is responsible for deciding the power set\hh points of the generation units to meet the demand locally and ensure an economic operation.
Usually, the \ac{ems} provides the power set\hh points on a time scale ranging between minutes to fraction of an hour, guaranteeing safe operation without violation of power and energy limits of the units.
On the lower control layers, often droop\hh based control strategies are used to maintain power balance (see, \eg , \cite{SchOrtAstRaiSez2014}).
This is shown in Figure \ref{fig:mg}.

\begin{figure}
  \centering
  \includegraphics{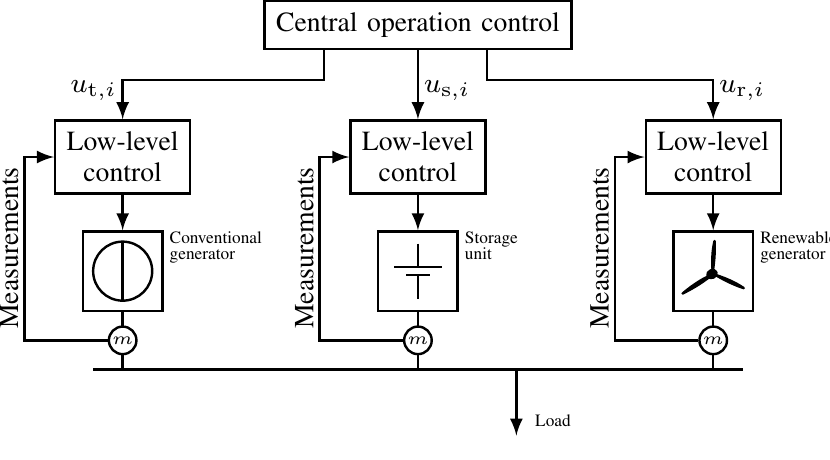}
  \caption{Block diagram representation: the central operation control provides power set\hh{}points to the units; the units change their power output based on a primary control law to ensure a power balance. The primary control laws of the different units can have saturation according to operation limits.}
  \label{fig:mg}
\end{figure}

Various control strategies to design \acp{ems} are available in literature.
Some of these strategies are reviewed in~\cite{ZiaElbBen2018}.
One popular strategy presented there is \acf{mpc}, which is a finite\hh horizon optimal control strategy that determines the control actions (\eg, power set\hh points) by solving an optimization problem.
At each time instance, an optimization problem is formulated based on renewable infeed forecast, demand forecast and the current state of the \ac{mg} and solved.
From the resulting sequence of control actions, only the first control action is applied to the system.
At the next time instance, the optimization problem is updated with the new state and forecasts and solved again.
A valuable feature of this control technique is its ability to account for constraints such as power and energy limits.
\ac{mpc}\hh based \acp{ems} are designed, \eg, in \cite{ParRikGli2014,HanNenRaiRei2014}.

Fluctuations in renewable infeed and load demand can compromise safe and reliable operation of \acp{mg}.
In literature there are \ac{mpc} designs that are capable of handling uncertainty (see, \eg,~\cite{SalOzkLudWeiHof2018,Lof2003,BemBorMor2003}).
In \cite{HanNenRaiRei2014}, the authors propose a \acf{mm} \ac{mpc} approach to address the uncertain load and renewable infeed.
This design considers the impact of droop control on the operation.

As opposed to an \ac{mpc}\hh \ac{ems} approach, pure droop schemes aiming for economic objectives are proposed in \cite{NutLohWanBla2015, NutLohWanBla2016,CheCheLiMenZheGueAbb2017}.
In \cite{NutLohWanBla2015}, a non\hh linear droop control law that prioritizes generation from units with less operation cost is suggested.
Furthermore, a strategy for tuning the droop gradient in accordance with the generation cost is proposed in \cite{NutLohWanBla2016}.
In \cite{CheCheLiMenZheGueAbb2017}, a linear droop scheme based on incremental cost of power generation is considered.
In \cite{DorSimBul2016}, it is stated that when the operation costs are strictly convex, selecting a suitable droop control 
law can lead to an economic operation.
One limitation of these approaches is their omission of storage dynamics and energy limits. Also, it is not possible to 
consider turning on and off the conventional generators.

A major drawback of the \ac{mm} \ac{mpc} proposed in \cite{HanNenRaiRei2014} is its conservativeness:
the power and energy limits are included as constraints and the power set\hh points are selected such that purely affine droop feedback would not violate these constraints for all possible disturbance realizations.
Thereby, this approach avoids saturation in the local unit controllers.
The limitation of avoiding actuator saturation in \ac{mpc} designs is recognized in \cite{CaoLin2005, HuaLiLinXi2011, OraBak2015}.
These works provide theoretical results to include this effect in feedback control of \ac{mm} \ac{mpc}.
However, the proposed approaches are not directly applicable to \acp{mg}.
For instance, they do not include binary decision variables that result from switching on and off conventional units and, in contrast to our \ac{mg} system, assume state feedback control.

Another limitation of most of the aforementioned power system control designs is their exclusion of \ac{res} in droop control.
These designs consider droop control only for storage and conventional units.
Increasing penetration of \ac{res} has encouraged developing droop strategies for \ac{res} (see, \eg, \cite{VanKooMeeVan2013,DasAltHanSorFlyAbi2018, NeeJohDelGonLav02016, XinLiuWanGanYan2013}).
They can increase the energy generated from renewable units, especially during periods with sufficient availability.
The benefit of droop control for renewable units in the reliable operation of a grid\hh connected \ac{mg} without energy storage is demonstrated in~\cite{XinLiuWanGanYan2013}. 
Similar results, including a practical study, are shown for an islanded scenario in~\cite{NeeJohDelGonLav02016}.
Such works encouraged us to extend the previous \ac{mm} \ac{ems} design (see \cite{HanNenRaiRei2014}) for systems with \ac{res} droop, expecting an improvement in the overall performance.
\subsection{Contributions}
The main contributions of this paper are as follows.
\begin{enumerate}
  \item We propose an \ac{mm} \ac{mpc}\hh based \ac{ems} that is both economic and robust.
  As compared to a related existing solution \cite{HanNenRaiRei2014}, we improve performance by including \ac{res} droop and droop saturation.
  \item The \ac{mm} \ac{mpc} with saturation corresponds to a robust mixed\hh integer program.
  Existing tractable reformulations (see, \eg, \cite{PauKinFaiGha2016}) are not directly applicable as the droop control with saturation constitutes a piecewise affine function.
  We derive a tractable reformulation of the \ac{mm} \ac{mpc} problem considered here.
  \item We show that the proposed controller increases the feasible set of controls over \ac{mm} \ac{mpc} without saturation.
  \item In a case study, we demonstrate performance improvement resulting from saturation and \ac{res} droop in terms of open\hh loop predicted cost as well as closed\hh loop cost.
\end{enumerate}

\subsection{Structure of the paper}
The structure of this paper is as follows.
Section~\ref{sec:microgrid:model} provides the mathematical model of an islanded \ac{mg}, droop control and saturation.
Section~\ref{sec:problem:formulation} defines the control objective and two \ac{mm} \ac{mpc} problems - one with \ac{res} droop and one with \ac{res} droop and saturation in all the units.
In Section~\ref{sec:solution}, we present a tractable reformulation of the \ac{mm} \ac{mpc} problem and also prove the increase of the feasible region.
Section~\ref{sec:casestudy} provides a case study comparing \ac{mm} \ac{mpc} controllers with and without saturation and with and without \ac{res} droop control.

\subsection{Mathematical preliminaries}
\label{sec:preliminaries}
The set of real numbers is denoted by $\RR$.
The sets of negative, positive, non\hh positive and non\hh negative real numbers are denoted by $\Rne$, $\Rp$, $\Rnz$ and $\Rpz$, respectively.
The set of positive integers is denoted by $\N$, and $\N_{[1, n]} = \{1, 2, \ldots, n\}$ is the set of the first $n$ positive integers.

The operator $\min(x, y)$ provides the element\hh wise minimum of the vectors $x$, $y$.
Similarly, $\max(x, y)$ provides the element\hh wise maximum.
If $x$ and $y$ are vectors or matrices of equal dimensions, then a comparison such as $x\leq y$ is true \ac{iff} all comparisons between elements at matching positions are true.
If $x$ is a matrix and $y$ is a vector, then a comparison is true \ac{iff} comparing every column in $x$ to $y$ according to the previous sentence evaluates to true.
If $x$ is a vector or matrix and $y$ is a scalar, then a comparison is true \ac{iff} comparing every single element in $x$ to $y$ evaluates to true.
If $x$ and $y$ are vectors or matrices of equal dimensions and $x\leq y$, then the set $[x,y]$ is the box determined by the intervals defined by elements at matching positions.

For a vector or matrix $x$, the transpose is denoted by $x\T$.
For a vector $x \in \RR^n$ with elements $x_i$, we use the notation $\mathbf{1}\T x = \sum_{i = 1}^{n} x_{i}$.
For $x \in \RR$ and $\delta \in \{0,1\}$, we define $y=\delta\wedge x$ as: $y=x$ if $\delta=1$ and $y=0$ if $\delta=0$.
When $x$ and $\delta$ are vectors, this operation is performed element\hh wise.
We sometimes use a dot to indicate that function arguments are omitted for brevity, \eg, writing $x(\cdot,\rho)$.

We explicitly point out that when we classify a function as \emph{increasing} or \emph{monotonically increasing}, respectively \emph{decreasing} or \emph{monotonically decreasing}, we allow it to be constant or constant in parts; \ie, we do \emph{not} imply \emph{strict} monotonicity.

%% file: model.tex

\section{Mirogrid Model}
\label{sec:microgrid:model}
In this section, we develop the mathematical model of an islanded \ac{mg} including droop control and saturation limits.
The model of the \ac{mg} and the notation are motivated from  \cite{HanNenRaiRei2014,HanSopRaiReiPat2020} and related works.
This model is used later for the \ac{ems} design.

For the model, we assume that the lower control layers (also referred to as low\hh level control in this paper), \ie, primary and secondary control ensure a stable operation of the \ac{mg}.
In addition, we assume that start\hh up and shut\hh down times of the conventional units are small compared to the \ac{ems} sampling time.
Furthermore, we assume that storage losses are negligible compared to the uncertainties posed by the \ac{res} and load demand.

\subsection{Notation}
We consider the \ac{mg} model with \ac{res} like wind turbines and \ac{pv} plants, battery storage units and conventional units like diesel generators.
Each unit in the \ac{mg} is provided with a power set\hh point from the \ac{ems}.
However, the uncertainty both in generation from the \ac{res} and in demand necessitates power outputs that differ from these set\hh points.

In the \ac{mg}, let us denote the number of conventional units, storage units and renewable units by $T$, $S$ and $R$, respectively, and the number of loads by $D$.
At a given sampling instance $k \in \Npz$, let us denote the disturbance by $w\ok = [w_\trs\ok\T ~ w_\tld\ok\T]\T$, where $w_\trs\ok \in \Rpz^R$ is the available renewable infeed and $w_\tld\ok \in \Rnz^D$ is the load demand.
The power set\hh points provided by the \ac{ems} are denoted by $u\ok = [u_\tth\ok\T ~ u_\tst\ok\T ~ u_\trs\ok\T ]\T$, where $u_\tth\ok\in \RR^T$, $u_\tst\ok\in\RR^S$ and $u_\trs\ok\in \RR^R$ are the set\hh points of the conventional units, storage units and \acs{res}, respectively.
Each conventional unit can be switched on or off, which is represented by a vector of binary variables $\delta_\tth\ok \in \{0, 1\}^{T}$.
The power output of the units is denoted by $p\ok = [p_\tth\ok\T ~ p_\tst\ok\T ~ p_\trs\ok\T]\T$, where $p_\tth\ok \in \Rpz^{T}$, $p_\tst\ok \in \RR^{S}$ and $p_\trs\ok \in \Rpz^{R}$.
The energy level of storage units is denoted by $x\ok \in \Rpz^{S}$.

\subsection{Microgrid without saturation and without renewable power sharing}
In an islanded \ac{mg}, the local generation should match the local consumption at all time instances.
This power balance condition can be represented by the algebraic equation
\begin{equation}\label{eq:power:balance}
 \mathbf{1}\T\ptk + \mathbf{1}\T\psk + \mathbf{1}\T\prk + \mathbf{1}\T\wlk = 0.
\end{equation}
The power provided by the renewable units depends on the available renewable infeed $\wrk$ and the power set\hh points $\urk$, \ie,
\begin{equation}\label{eq:renewable:power}
 \prk = \min(\urk, \wrk).
\end{equation}

The uncertain load $\wlk$ and renewable infeed $\prk$ cause mismatch in the power balance.
This mismatch is compensated by the storage and conventional units~\cite{SchOrtAstRaiSez2014}.
Each unit has a low\hh level droop control that ensures a desired proportional power sharing~(see, \eg, \cite{KriHanSchRaiKra2017}, \cite{SchHanKraOrtRai2017}).
The share of each unit depends on the inverse droop gain, which we denote by $\chi_{\tth} \in \Rpz^{T}$ for the conventional units and $\chi_{\tst} \in \Rpz^{S}$ for the storage units.
These inverse droop gains can be chosen, \eg, according to the nominal power of the units.
Let us denote $\chi = [\chi_{\tth}\T ~ \chi_{\tst}\T]\T$.

\begin{subequations}\label{eq:storage:model}
Consider an auxiliary free variable $\rhok \in \RR$.
In steady-state, the power of the storage units can be described by
\begin{align}\label{eq:storage:power}
  \psk = \usk + \chi_{\tst}\,\rhok.
\end{align}
The dynamics of the storage units are
  \begin{align}\label{eq:storage:dynamics}
   x\ok = x\okm - \Ts\,\psk,
  \end{align}
  where $\Ts \in \Rp$ is the sampling time.
  The energy storage capacities are included by the constraints
  \begin{align}\label{eq:storage:limits}
   \xsmin \leq x\ok \leq \xsmax,
  \end{align}
\end{subequations}
with $\xsmin \in \Rpz^{S}$ and $\xsmax \in \Rp^{S}$.

Conventional units can be switched on or off.
When conventional unit $i\in\N_{[1,T]}$ is switched off, \ie, $\deltatik = 0$, then its power $\ptik = 0$ and it cannot participate in power sharing.
When switched on, \ie, $\deltatik = 1$, it participates in power sharing. This behavior can be modeled by
\begin{equation}\label{eq:conventional:power}
\ptk = \deltatk \wedge (\utk + \chi_{\tth}\,\rhok).
\end{equation}

\begin{subequations}\label{eq:power:limits}
The constraints on the power set\hh points and the power of the units are
  \begin{alignat}{2}
    \umin &\leq \uk & &\leq\umax,\label{eq:uminmax}\\
    \begin{bmatrix}\deltatk\wedge\ptmin\\ \psmin\\ \prmin\end{bmatrix} &\leq \pk & &\leq \pmax,\label{eq:pminmax}
  \end{alignat}
\end{subequations}
with $\umin\in\RR^{T+S+R}$, $\umax\in\RR^{T+S+R}$, $\pmin\in\Rpz^T\times\Rne^S\times\Rpz^R$ and $\pmax\in\Rp^{T+S+R}$.
We subdivide these limits in the same manner as $u\ok$ and $p\ok$, \eg, $\pmin=[(\ptmin)\T~(\psmin)\T~(\prmin)\T]\T$.

\subsection{Microgrid with renewable power sharing}

In the previous model, we have restricted power sharing to storage and conventional units.
However, as we already considered the limitation due to available renewable power, we can also include renewable units in power sharing.
Let us define $\chi_{\trs} \in \Rpz^{R}$ and redefine the vector of inverse droop constants by $\chi = [\chi_{\tth}\T ~ \chi_{\tst}\T~\chi_{\trs}\T]\T$.
Then the renewable power \eqref{eq:renewable:power} is redefined as
\begin{equation}\label{eq:prsokNoSatAtMin}
 \prk = \min(\urk + \chi_\trs\,\rhok, \wrk).
\end{equation}

\subsection{Microgrid with saturation}
We consider saturation as a hard limiter enforcing the physical operation range of a unit.
Therefore, the power output of units with droop control given by~\eqref{eq:storage:power}, \eqref{eq:conventional:power}, \eqref{eq:prsokNoSatAtMin} is now expressed as a feedback law of $\rhok$ and saturation.
Let us define saturation of a variable (for example, $p$) as
\begin{equation}\label{eq:saturation}
\sat(\pmin, p, \pmax) := \begin{cases}
\pmin , & \text{if } p < \pmin, \\
p,  & \text{if } p \in [\pmin, \pmax], \\
\pmax,  & \text{if } p > \pmax,
\end{cases}
\end{equation}
where $\pmin\leq\pmax$.
When $p$, $\pmin$ and $\pmax$ are vectors, the $\sat(\cdot, \cdot, \cdot)$ operator is understood element\hh wise.

With saturation, operation constraints are imposed by limiting the output power at the lower control layer.
Now the renewable power \eqref{eq:prsokNoSatAtMin} is redefined as
\begin{equation}\label{eq:prsok}
 \prk = \sat(\prmin, \urk + \chi_\trs\,\rhok, \wrk).
\end{equation}
Note that \eqref{eq:prsokNoSatAtMin} already includes saturation at the upper limit.

For conventional generators, the power given by~\eqref{eq:conventional:power} is redefined as
\begin{equation}\label{eq:conventional:saturation:power}
\ptk = \deltatk \wedge \sat(\ptmin, \utk + \chi_{\tth}\,\rhok, \ptmax).
\end{equation}

Operation of storage units is restricted by power as well as energy limits.
A straightforward way to implement energy\hh based saturation would involve setting the power to zero in the moment the energy reaches one of the bounds in~\eqref{eq:storage:limits}.
However, to avoid such sudden power changes, which could happen also between sampling instances of the \ac{ems}, and to keep the analysis simple, we choose a different approach.
Based on the sampling time $\Ts$ and current energy level $x$, dynamically adjusted power limits are determined by%
\begin{subequations}\label{eq:pstok}\begin{equation}\label{eq:psiminmaxk}\begin{split}
 \psmink &= \max\left(\psmin, \frac{x\okm-x^\tmax}{\Ts}\right),\\
 \psmaxk &= \min\left(\psmax, \frac{x\okm-x^\tmin}{\Ts}\right),
\end{split}\end{equation}
and the power is subject to saturation at these limits, \ie,
\begin{equation}\label{eq:pstokMain}
  \psk = \sat(\psmink, \usk+\chi_{\tst}\,\rhok, \psmaxk).
\end{equation}\end{subequations}

\remark{When the power sharing constant of a unit is $\chi_i = 0$, for $i\in \N_{[1, T+S+R]}$, then the unit does not participate in power sharing.
This means flexibility to decide which units can participate in power sharing.}

\remark{Please note that in the saturation\hh based model, the power set\hh point limits $\umin$ and $\umax$ in \eqref{eq:uminmax} can differ from and even lie beyond the power limits $\pmin$ and $\pmax$.}

\remark{Also note that now the explicit energy and power constraints~\eqref{eq:storage:limits} and \eqref{eq:pminmax} are no longer required.
Instead, the limits $\xsmin$, $\xsmax$, $\pmin$ and $\pmax$ are implicitly enforced by saturation.}

%% file: problem.tex

\section{MINIMAX MPC FOR MICROGRID WITH SATURATION}
\label{sec:problem:formulation}
The goal is to design an \ac{mpc}\hh based \ac{ems} that is robust \wrt uncertain available renewable infeed and load demand.
We consider an \ac{mm} formulation where the worst\hh case operation cost of the \ac{mg} over all possible disturbance realizations is minimized~\cite{Lof2003}.
In this section, we define the uncertainty, the operation cost and later formulate the \ac{mm} \ac{mpc} problem.

\subsection{Uncertainty model}
For the considered renewable units, the available power depends on the weather conditions (\ie, solar irradiation for \ac{pv} plants and wind speed for wind turbines).
Using historic data and a forecaster, future lower and upper bounds for the available power can be derived (see, \eg,~\cite{HynAth2018}).
Let us denote these bounds at time instance $k$ by
\begin{equation}\label{eq:renewable:uncertain}
w_{\trs}^{\tmin}\ok \leq \wrk \leq w_{\trs}^{\tmax}\ok,
\end{equation}
where $w_{\trs}^{\tmin}\ok \in \Rpz^R$ and $w_{\trs}^{\tmax}\ok \in \Rpz^R$.
Similarly, the minimum and maximum bounds for the load demands are
\begin{equation}\label{eq:load:uncertain}
w_{\tld}^{\tmin}\ok \leq \wlk \leq w_{\tld}^{\tmax}\ok,
\end{equation}
where $w_{\tld}^{\tmin}\ok \in \Rnz^D$ and $w_{\tld}^{\tmax}\ok \in \Rnz^D$.
Based on \eqref{eq:renewable:uncertain}, \eqref{eq:load:uncertain} we can pose
\begin{equation}
w^{\tmin}\ok \leq w\ok \leq w^{\tmax}\ok,
\end{equation}
where $\wmink = [w_\trs^{\tmin}\ok\T ~ w_\tld^{\tmin}\ok\T]\T$ and $\wmaxk = [w_\trs^{\tmax}\ok\T ~ w_\tld^{\tmax}\ok\T]\T$.

\subsection{Operation cost}
The operation cost under consideration is economically motivated.
We assume there is no cost on operating renewable units.
The operation cost of conventional units includes fuel cost, fixed\hh generation cost and switching cost, \ie,
\begin{multline}\label{eq:stage:cost:thermal}
  \ell_{\tth}(\ptk, \deltatk, \deltatkm) = \Cth\T\,\ptk + \CthOn\T\,\deltatk +\\
  \CthSw\T \left|\deltatk-\deltatkm\right|,
\end{multline}
where $\Cth, \CthOn, \CthSw \in \Rpz^{T}$.

Typically, the purpose of storage units is to store the excess energy for future usage.
This can be encouraged by including a cost on the storage power, \ie,
\begin{equation}\label{eq:stage:cost:storage}
\ell_{\tst}(p_{\tst}\ok) = \Cst\T\,\psk,
\end{equation}
where $\Cst \in \Rpz^{S}$.
Hence, $\ell_{\tst}$ is negative if power is stored, \ie, if $\psk$ is negative.
In particular, this cost discourages wasting available renewable power just because it cannot be consumed by load demand instantaneously.

The total operation cost of an \ac{mg} is the sum of \eqref{eq:stage:cost:thermal} and \eqref{eq:stage:cost:storage}, \ie,
\begin{multline}\label{eq:stage:cost}
  \ell(p\ok, \deltatk, \deltath\okm) = \ell_{\tth}(\ptk, \deltatk, \deltatkm) +\\
  \ell_{\tst}(\psk).
\end{multline}

\subsection{Minimax MPC}
In certainty equivalence \ac{mpc}, the power set-points are determined by minimizing the operation cost over the prediction horizon for a given disturbance realization.
Therefore, this formulation does not rigorously account for uncertainties. In a \acl{mm} strategy, uncertainties are handled by considering their worst\hh case impact~\cite{Lof2003}.
More precisely, a control (here, power set\hh points and switch statuses) is determined such that the operation cost is minimized over robustly feasible controls and at the same time maximized over possible disturbance realizations.
Robustly feasible controls guarantee constraints to be satisfied for all possible disturbance realizations.

Let us consider the prediction horizon of the \ac{mm} \ac{mpc} as $\Npr \in \N$.
At sampling instance $k$, the power predicted at future step $j \in\Npr$ is given by $p\okj$.
Let us define matrices to represent profiles of variables over the prediction horizon as
\begin{equation}\label{eq:defineVec}\begin{split}
 \deltathVec &:= [\deltath\okk~\cdots~\deltath\okNp],\\
 \uVec &:= [u\okk~\cdots~u\okNp],\\
 \pVec &:= [p\okk~\cdots~p\okNp],\\
 \xVec &:= [x\okk~\cdots~x\okNp],\\
 \wVec &:= [w\okk~\cdots~w\okNp],\\
 \wminVec &:= [\wmin\okk~\cdots~\wmin\okNp],\\
 \wmaxVec &:= [\wmax\okk~\cdots~\wmax\okNp].
\end{split}\end{equation}
Finally, let us define the operation cost over the prediction horizon as 
\begin{equation}\label{eq:operationalCost}
 J(\pVec, \deltathVec, \deltatk) := \sum_{j = 1}^{\Npr} \ell(p(k + j), \deltath(k+j), \deltath(k + j- 1)).
\end{equation}

Using the notation introduced above, we first formulate the minimax \ac{mpc} which includes renewable droop but 
no saturation in storage and conventional units.
\begin{problem}[\acs{mm} \ac{mpc} with \acs{res} droop]\label{prob:minimax:res}
\begin{subequations}\label{eq:minimax:res}
  \begin{align}
    \min_{\uVec,\,\deltathVec}\ \max_{\wVec}\ J(\pVec, \deltathVec, \deltatk),
  \end{align}
  where $J\odo$ is defined by~\eqref{eq:operationalCost}, subject to the model equations
  \begin{align}
  \begin{split}\label{eq:minimax:model:res}
      0 &= \mathbf{1}\T p(k + j) + \mathbf{1}\T w_\tld(k +j),\\
      \prs\okj &= \min(\urs\okj + \chi_{\trs}\,\rho\okj, \wrs\okj),\\
      \pth\okj &= \deltath\okj \wedge (\uth\okj + \chi_{\tth}\,\rho\okj),\\
      \pst\okj &= \ust\okj+\chi_{\tst}\,\rho\okj,\\
      x\okj &= x\okjm - \Ts\,\pst\okj,\\
      j &\in \N_{[1, \Npr]},
  \end{split}
  \shortintertext{initial conditions}
   x_\tst(k) &= x_{\tst, 0},\quad \deltath(k) = \delta_{\tth, 0},\label{eq:minimax:ic}\\
  \shortintertext{control constraints}
  \begin{split}\label{eq:minimax:control}
   \deltathVec &\in \{0, 1\}^{T\times \Npr},\\
   \umin &\leq \uVec \leq \umax,
  \end{split}
  \shortintertext{uncertainty model}
   \wminVec &\leq \wVec \leq \wmaxVec,\label{eq:minimax:uncert}\\
  \shortintertext{power and energy constraints}
  \begin{split}\label{eq:minimax:power}
   \begin{bmatrix}\deltathVec\wedge\ptmin\\ \psmin\\ \prmin\end{bmatrix} &\leq \pVec \leq \pmax,\\
   x^\tmin &\leq \xVec \leq x^\tmax,
   \end{split}
   \end{align}
   \end{subequations}
and the condition that the control $\deltathVec$, $\uVec$ must be feasible \wrt \eqref{eq:minimax:model:res}, \eqref{eq:minimax:ic}, \eqref{eq:minimax:power} 
$\forall\,\wVec\in[\wminVec,\wmaxVec]$.
\end{problem}

\vspace{0.3em}

We now formulate the \acl{mm} \ac{mpc} with droop and saturation at all the units.
\begin{problem}[\acs{mm} \ac{mpc} with \acs{res} droop and saturation]\label{prob:minimax}
\begin{subequations}\label{eq:minimax}
  \begin{align}
    \min_{\uVec,\,\deltathVec}\ \max_{\wVec}\ J(\pVec, \deltathVec, \deltatk),
  \end{align}
  where $J\odo$ is defined by~\eqref{eq:operationalCost}, subject to the model equations
  \begin{equation}\label{eq:minimax:model}\begin{split}
      0 &= \mathbf{1}\T p(k + j) + \mathbf{1}\T w_\tld(k +j),\\
      \prs\okj &= \makebox[0em][l]{\ensuremath{\sat(\prmin, \urs\okj +}}\begin{multlined}[t]\\
       \chi_{\trs}\,\rho\okj, \wrs\okj),\end{multlined}\\
      \pth\okj &= \makebox[0em][l]{\ensuremath{\deltath\okj \wedge \sat(\ptmin, \uth\okj +}}\begin{multlined}[t]\\
       \chi_{\tth}\,\rho\okj, \ptmax),\end{multlined}\\
      \psminn\okj &= \max(\psmin, (x\okjm-x^\tmax)\,\Ts^{-1}),\\
      \psmaxx\okj &= \min(\psmax, (x\okjm-x^\tmin)\,\Ts^{-1}),\\
      \pst\okj &= \makebox[0em][l]{\ensuremath{\sat(\psminn\okj, \ust\okj+}}\begin{multlined}[t]\\
       \chi_{\tst}\,\rho\okj, \psmaxx\okj),\end{multlined}\\
      x\okj &= x\okjm - \Ts\,\pst\okj,\\
      j &\in \N_{[1, \Npr]},
  \end{split}\end{equation}\end{subequations}
initial conditions~\eqref{eq:minimax:ic}, control constraints~\eqref{eq:minimax:control} and uncertainty model~\eqref{eq:minimax:uncert}, and the condition that the control $\deltathVec$, $\uVec$ must be feasible \wrt \eqref{eq:minimax:model}, \eqref{eq:minimax:ic} $\forall\,\wVec\in[\wminVec,\wmaxVec]$.
\end{problem}

\vspace{0.3em}

\begin{remark}
In Problem~\ref{prob:minimax} the constraints on the unit power and the storage unit energy are included in the saturation function and there is no need to write them explicitly.
\end{remark}

In general, minimax problems are hard to solve.
A tractable reformulation for robust integer problems without considering binary decisions in feedback is provided in~\cite{PauKinFaiGha2016}.
But the saturation in the above problem leads to binary variables in the feedback.
In previous work \cite{HanNenRaiRei2014}, a tractable reformulation has been found with binaries affecting the feedback.
In that work, both robust feasibility and the inner maximization problem only require checking two disturbance cases -- one where it is minimum at all times and one where it is maximum at all times.
However, the analysis presented in~\cite{HanNenRaiRei2014} is not directly applicable to the above problems as it does not include \ac{res} droop and droop saturation.
Therefore, it is necessary to derive tractable reformulations for the above problems.

%% file: solution.tex

\section{Tractable reformulation for the minimax mpc with saturation}\label{sec:solution}

In this section, to keep the presentation simple, we show the detailed derivation of the tractable reformulation for only Problem~\ref{prob:minimax}.
We show that the maximum operation cost occurs for minimum disturbance $\wVec=\wminVec$.
Furthermore, we show that ensuring that constraints are satisfied for $\wVec=\wminVec$ and $\wVec=\wmaxVec$ is sufficient to ensure that they are satisfied for all possible disturbance realizations.
We also show that the \ac{mm} \ac{mpc} problem with saturation has a larger feasible control region than the \ac{mm} \ac{mpc} problem without saturation.

\subsection{Tractable formulation}
In Problem~\ref{prob:minimax}, the switch status of the conventional generators $\deltathVec$ makes the outer minimization a mixed\hh integer problem.
However, the disturbance $\wVec$ cannot directly modify the switch status, which makes the inner maximization problem integer\hh free.
Furthermore, note that all droop gains $\chi_{\tst}$, $\chi_{\tth}$, $\chi_{\trs}$ are non\hh negative.
Hence the power values of all units which are not yet in saturation either increase or decrease simultaneously to achieve the power balance, or stay constant if the respective droop gain is zero.
The following theorems pinpoint the disturbance sequences that have to be considered (instead of all possible sequences) with regard to robust feasibility and the maximum cost.
\begin{theorem}\label{thm:constraint}
In Problem~\ref{prob:minimax}, given non\hh negative droop gains $\chi_{\tst}$, $\chi_{\tth}$, $\chi_{\trs}$, the set of feasible controls ($\deltathVec$, $\uVec$) reduces to the set of controls which are feasibile for disturbance realizations $\wVec=\wminVec$ and $\wVec=\wmaxVec$.

  \proof See Appendix~\ref{sec:proofof:lemmas}~and~\ref{sec:proofof:thm:constraint}.
\end{theorem}
\begin{theorem}\label{thm:cost}
  In Problem~\ref{prob:minimax}, given non\hh negative droop gains $\chi_{\tst}$, $\chi_{\tth}$, $\chi_{\trs}$, the worst\hh case operation cost corresponds to the disturbance realization $\wVec=\wminVec$.

  \proof See Appendix~\ref{sec:proofof:lemmas}~and~\ref{sec:proofof:thm:cost}.
\end{theorem}

Using Theorem~\ref{thm:constraint} and~\ref{thm:cost}, the \ac{mm} \ac{mpc} Problem~\ref{prob:minimax} can be equivalenty stated as
\begin{problem}[Tractable \ac{mm} \ac{mpc}]\label{prob:tractable:minimax}
  \begin{align}
    \min_{\uVec,\,\deltathVec}\ J(\pVec, \deltathVec, \deltatk)_{\left| \wVec = \wminVec \right.},
  \end{align}
  where $J\odo$ is defined by~\eqref{eq:operationalCost}, subject to the model equations~\eqref{eq:minimax:model}, initial conditions~\eqref{eq:minimax:ic}, control constraints~\eqref{eq:minimax:control}, and the condition that the control $\deltathVec$, $\uVec$ must be feasible \wrt \eqref{eq:minimax:ic}, \eqref{eq:minimax:model} for $\wVec=\wminVec$ and for $\wVec=\wmaxVec$.
\end{problem}

\begin{remark}
The $\sat(\cdot,\cdot,\cdot)$ in the current formulation can be included with additional binary varibales as in~\cite{BemMor1999}, which may increase computational complexity.
\end{remark}

\subsection{Enhancing feasibility}
The robust feasibility requirement of the \ac{mm} \ac{mpc} optimization problem restricts the set of admissible controls, which leads to conservative power set\hh points.
In Problem~\ref{prob:minimax}, constraints on the power and energy are included as saturation instead of hard limits $\pVec\in[\pmin,\pmax]$, $\xVec\in[\xsmin,\xsmax]$ as in Problem~\ref{prob:minimax:res}.
It is shown here that this enlarges the feasible control space.

\begin{theorem}\label{thm:feasibility}
Any feasible solution to Problem~\ref{prob:minimax:res} is a feasible solution to Problem~\ref{prob:minimax}.
However, the converse does not hold in general.
Furthermore, a feasible solution for Problems~\ref{prob:minimax:res} and~\ref{prob:minimax} is associated with the same cost value in both problems.
\end{theorem}
\begin{proof}
An alternative way to express Problem~\ref{prob:minimax:res} is by augmenting Problem~\ref{prob:minimax} with the following constraints:
\begin{equation*}\begin{split}
 \urs\okj + \chi_\trs\,\rho\okj &\geq \prmin,\\
 \ust\okj + \chi_\tst\,\rho\okj &\in [\psmin, \psmax],\\
 x\okjm - \Ts\left(\ust\okj + \chi_\tst\,\rho\okj\right) &\in [\xsmin,\xsmax],\\
 \deltatkj \wedge (\ust\okj + \chi_{\tth}\,\rho\okj) &\in [\ptmin,\ptmax].
\end{split}\end{equation*}
Essentially, these are constraints preventing units from going into saturation.
Therefore, a feasible solution to the more stringent Problem~\ref{prob:minimax:res} is a feasible solution to Problem~\ref{prob:minimax}. 
As the problems differ only in constraints, not in the cost function, the feasible solution is associated with the same cost value in both problems.

But the vice\hh versa does not hold in general.
As a counter\hh example, consider a scenario when the power set\hh points of units $1$ and $2$ are $u_{1}\okj = p_{1}^\tmin$ and $u_{2}\okj = p_{2}^\tmax$.
Hence, without saturation, $p_{1}\okj = p_{1}^\tmin + \chi_{1}\,\rho\okj$ and $p_{2}\okj = p_{2}^\tmax + \chi_{2}\,\rho\okj$.
The power constraints $[p^\tmin ~p^\tmax]$ in Problem~\ref{prob:minimax:res} can be satisfied \acl{iff} $\rho\okj = 0$, \ie, a power imbalance caused by an unknown disturbance in the system cannot be compensated.
On the other hand, when units are allowed to saturate, $\rho\okj$ can become positive and unit $1$ can increase its power output, $p_{1}\okj>p_{1}^\tmin$, to compensate a negative power imbalance, while the power output of unit $2$ stays constant at $p_{2}\okj=p_{2}^\tmax$.
Similarly, a positive imbalance can be compensated by unit $1$ keeping its power at $p_{1}\okj=p_{1}^\tmin$ and unit $2$ decreasing its power, $p_{2}\okj<p_{2}^\tmax$.
\end{proof}

%% file: caseStudy.tex

\section{CASE STUDY}\label{sec:casestudy}
In this section, we demonstrate the benefits of the proposed saturation\hh based robust controller assuming the \ac{mg} with high share of renewable infeed shown in Figure~\ref{fig:flatMercury}.
The \ac{mg} consists of a wind turbine, a \ac{pv} power plant, a storage unit, a conventional unit and a load.
Power and energy limits of the units are posed in Table~\ref{tab:casestudy} together with the droop gains and the weights of the cost function.
Note that $c_{\tst}$ is smalller than $c_{\tth}$ to discourage charging the storage with conventional power.

\begin{figure}[t]
	\centering
	\includegraphics{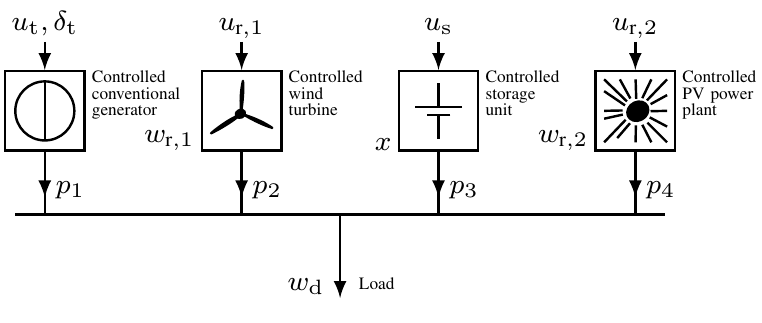}
	\caption{Test microgird topology}
	\label{fig:flatMercury}
\end{figure}

\begin{table}[t]
\caption{Unit parameters and weights of cost function.}\label{tab:casestudy}
\centering
 \begin{tabularx}{0.95\columnwidth}{lllX}
 \toprule
  Parameter & Value & Weight & Value\\
 \cmidrule(r){1-2}\cmidrule(l){3-4}
  $[\utmin~\usmin~\urwindmin~\urpvmin]$ & $[{-5}~{-5}~{-5}~{-5}]\,\unit{pu}$ & $\Cth$ & $1$\\
  $[\utmax~\usmax~\urwindmax~\urpvmax]$ & $[5~5~5~5]\,\unit{pu}$ & $\CthOn$ & $0.2$\\
  $[\ptmin~\psmin~\prwindmin~\prpvmin]$ & $[0.2~{-1}~0~0]\,\unit{pu}$ & $\CthSw$ & $0.3$ \\
  $[\ptmax~\psmax]$ & $[1~1]\,\unit{pu}$ & $\Cst$ & $0.9$ \\
  $[x_\tmin~x_\tmax]$ & $[0~6]\,\unit{pu\,h}$ & & \\
  $x^0$ & $2\,\unit{pu\,h}$ & & \\
  $[\chi_\tth~\chi_\tst~\chirwind~\chirpv]$ & $[1~1~1~1]$ & & \\
 \bottomrule
 \end{tabularx}
\end{table}

The \ac{ems} sampling time is chosen to be $15\,\unit{min}$ and the prediction horizon of the \ac{mpc} is $8\,\unit{h}$, \ie, $N_p = 32$. The simulation horizon is $6\,\unit{days}$, \ie, $\Nsi=576$.
The model and the controllers are implemented in Matlab\textsuperscript{\textregistered}2015a.
The \ac{mpc} optimization is formulated using YALMIP~\cite{Lof2004} and solved with Gurobi~{8.1.1}~\cite{gurobi}.

The available renewable infeed is generated based on real measurement data provided by the Atmospheric Radiation Measurement (ARM) Climate Research Facility \cite{ARM2011}, located at Graciosa Airport, Azores, Portugal.
The robust intervals for the disturbance are typically generated using a forecaster~\cite{KlaHan2018, HynAth2018}.
As this generation is out of the scope of the paper, we assume hypothetical robust intervals given by minimum and maximum disturbance, $w^\tmin\ok$ and $w^\tmax\ok$.
The resulting robust intervals are highlighted in Figure~\ref{fig:closedLoopSimulation}.
The robust interval for the load demand was generated in a similar fashion and is also illustrated in Figure~\ref{fig:closedLoopSimulation}.

For the simulations, we consider the \emph{worst\hh case} disturbance realization, $\wVec=\wminVec$, as this corresponds to the worst\hh case open\hh loop operation cost (see Theorem~\ref{thm:cost}).

\subsection{Prescient controller}
A prescient \ac{mpc} is a hypothetical controller which has prefect future knowledge of the available renewable infeed and the load demand.
In particular, for the worst\hh case disturbance realization $\wVec^\tmin$, the corresponding worst\hh case prescient \ac{mpc} is formulated based on Problem~\ref{prob:minimax} as follows.
The uncertainty model \eqref{eq:minimax:uncert} is replaced by $\wVec=\wminVec$ and, accordingly, the robust feasibility condition is replaced by the condition of feasibility for $\wVec=\wminVec$.
The closed\hh loop simulation with worst\hh case disturbance realization and worst\hh case prescient \ac{mpc} is visualized in Figure~\ref{fig:closedLoopSimulation}.
The corresponding open\hh loop cost values predicted each time the worst\hh case prescient \ac{mpc} problem is solved are included in Figure~\ref{fig:predictedObjective}.

\begin{remark}\label{rem:prescient:lower}
The open\hh loop cost value corresponding to the worst\hh case prescient controller is a lower bound for the cost value corresponding to the \ac{mm} \ac{mpc} Problem~\ref{prob:minimax}.
This can be easily deduced from Theorems~\ref{thm:constraint}~and~\ref{thm:cost}.
In fact, the open\hh loop cost of the prescient \ac{mpc} is a lower bound for the open\hh loop cost value corresponding to any conceivable robust \ac{mpc} controller.
This can be inferred by noting that the prescient \ac{mpc} is able to select an optimal physically possible power output profile $p\okj$, $\jInOneNp$, and a corresponding switch status profile.
\end{remark}

\subsection{Open-loop comparison}
Here, we compare open\hh loop performance of different \ac{mpc} controllers for given initial energy levels and switch statuses and identical robust forecast intervals.
A collection of $576$ variations of this data is obtained from the closed\hh loop simulation with the worst\hh case prescient \ac{mpc}.

Here, we define the labels for the different considered \ac{mpc} controllers:
\begin{enumerate}
 \item \PrescientMpc: worst\hh case prescient \ac{mpc},
 \item \MinimaxMpcResDroopOffDroopSatOff : Problem~\ref{prob:minimax:res} with $\chi_{\trs} = 0$,
 \item \MinimaxMpcResDroopOffDroopSatOn: Problem~\ref{prob:minimax} with $\chi_{\trs} = 0$,
 \item \MinimaxMpcResDroopOnDroopSatOff: Problem~\ref{prob:minimax:res} with $\chi_{\trs} > 0$,
 \item \MinimaxMpcResDroopOnDroopSatOn: Problem~\ref{prob:minimax} with $\chi_{\trs} > 0$.
\end{enumerate}

Figure~\ref{fig:predictedObjective} shows the open\hh loop cost predicted by these controllers.
We can observe:
\begin{itemize}
 \item \MinimaxMpcResDroopOffDroopSatOn \ac{mpc} does not reduce the predicted cost value compared to \MinimaxMpcResDroopOffDroopSatOff \ac{mpc}.
 \item \MinimaxMpcResDroopOnDroopSatOff \ac{mpc} reduces the cost value compared to \MinimaxMpcResDroopOffDroopSatOff \ac{mpc} during periods of high share of renewable infeed, \eg, between days $4$ and $6$.
 \item \MinimaxMpcResDroopOnDroopSatOn \ac{mpc} provides an equal or lower cost than the remaining controllers.
 This illustrates that \MinimaxMpcResDroopOnDroopSatOn \ac{mpc} increases the feasible region compared to \MinimaxMpcResDroopOnDroopSatOff \ac{mpc} and thereby achieves a lower open\hh loop cost (Theorem~\ref{thm:feasibility}).
\end{itemize}
These items suggest a synergy between the features \ac{res} droop and saturation of \emph{all} units.
In fact, the cost value of \MinimaxMpcResDroopOnDroopSatOn \ac{mpc} and the worst\hh case prescient \ac{mpc} apparently coincide over the entire simulation horizon.
This indicates that no conceivable robust \ac{mpc} and low\hh level control combination could perform better in this case study as far as open\hh loop prediction is concerned (see Remark~\ref{rem:prescient:lower}).

\begin{figure}[htp]
  \centering
  \includegraphics{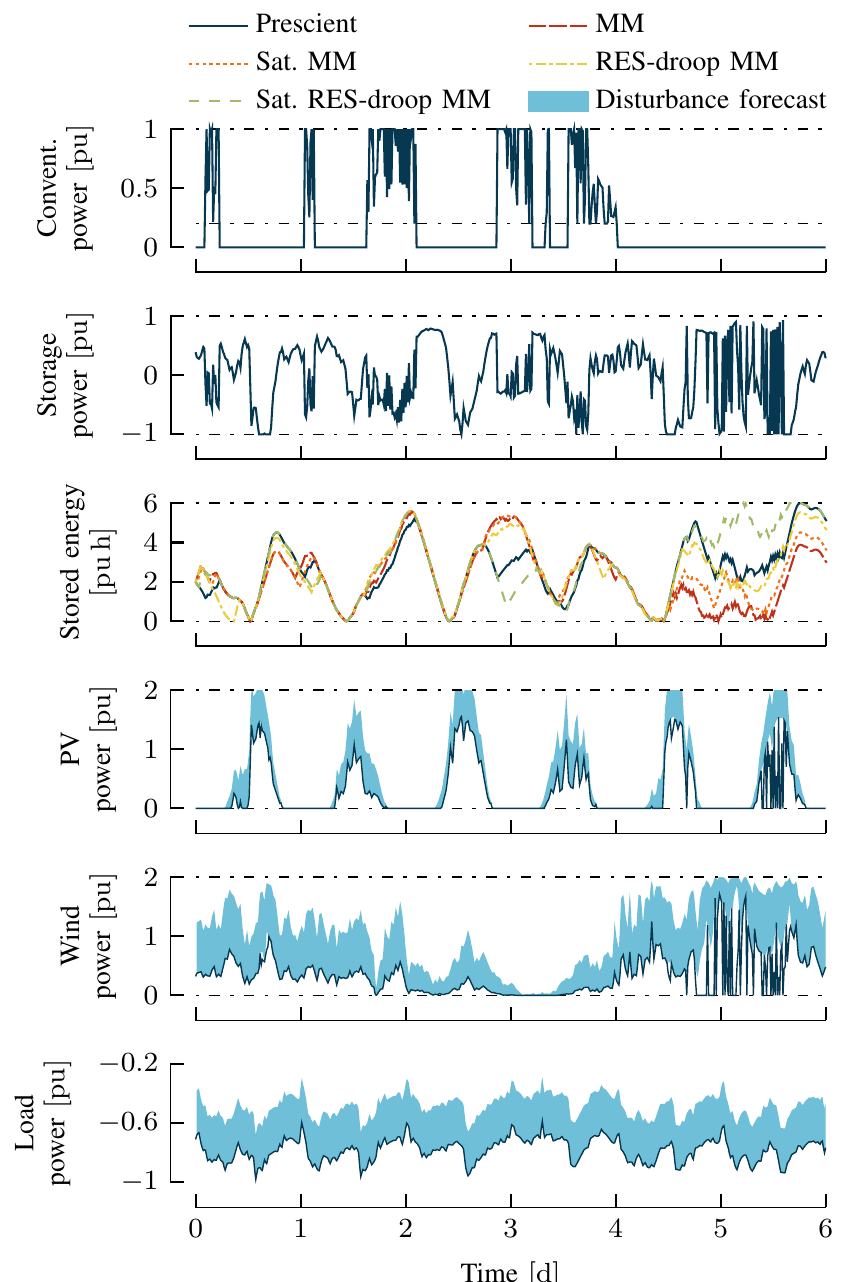}
  \caption{Closed\hh loop simulation with the worst case disturbance realization.
  For clarity, the plots corresponding to controllers other than prescient \ac{mpc} are only shown in the case of stored energy.}
  \label{fig:closedLoopSimulation}
\end{figure}

\begin{figure}[htp]
  \centering
  \includegraphics{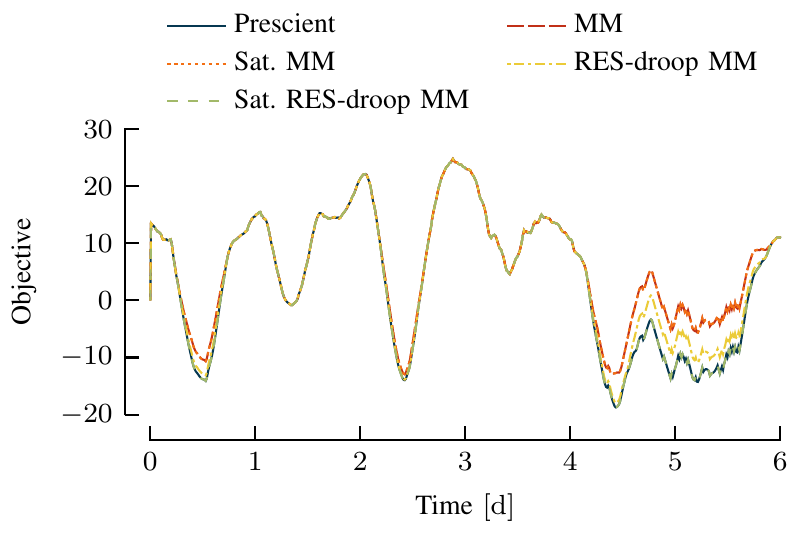}
  \caption{Predicted \ac{mpc} cost for $576$ different initial conditions and robust forecast intervals.
  The sequence of variations of this data stems from a closed\hh loop simulation with the worst\hh case prescient controller.}
  \label{fig:predictedObjective}
\end{figure}

The average (per sample over the $6$\hh day period) predicted cost and average predicted \ac{res} and conventional infeed of the controllers are shown in Table~\ref{tab:closedOpenLoop}.
Here, differences in average predicted \ac{res} infeed energies can be observed, while average predicted conventional infeed is identical for all \ac{mpc} controllers.
This indicates that lower predicted cost values are predominantly attributed to harvesting more energy from \ac{res} thus boosting storage energy level at the end of the prediction horizon.

\makeatletter\let\expandableinput\@@input\makeatother

\begin{table}[t]
\caption{Comparison of average open\hh loop \ac{mpc} performance for $576$ initial conditions (per\hh sample).}\label{tab:closedOpenLoop}
\centering
 \begin{tabularx}{0.95\columnwidth}{lXXX}
 \expandableinput closedOpenLoop-dataVariation-1
 \end{tabularx}
\end{table}

\begin{table}[t]
\caption{Comparison of closed\hh loop performance for worst\hh case disturbance realization (per\hh sample except switchings).}\label{tab:closedLoop}
\centering
 \begin{tabularx}{0.95\columnwidth}{lXXXX}
 \expandableinput closedLoop-dataVariation-1
 \end{tabularx}
\end{table}

\subsection{Closed-loop simulation}
Closed\hh loop simulations with different \ac{mpc} controllers are performed over a $6$\hh day period, comprising $\Nsi=576$ sampling instances.
These closed\hh loop simulations are performed assuming the \emph{worst\hh case} disturbance realization $\wk=\wmink$.
Figure~\ref{fig:closedLoopSimulation} includes energy profiles resulting for different \ac{mpc} controllers.
Here, we can observe that the \MinimaxMpcResDroopOnDroopSatOn \ac{mpc} tends to harvest more energy resulting in higher energy levels.

Based on the total operation cost \eqref{eq:stage:cost} and the cost definition in Problem~\ref{prob:minimax}, closed\hh loop per\hh sample cost values are calculated over the entire $\Nsi=576$ samples in the simulation by
\begin{equation}
 J^\text{closed\hh loop}:=\nicefrac{1}{\Nsi}\textstyle\sum_{k=1}^{\Nsi}\ell\left(p\ok,\deltath\ok,\deltath\okm\right),
\end{equation}
where $\deltath(0)$ are the given initial switch statuses and $p\ok$, $\deltath\ok$, $k\in\{1,\ldots,\Nsi\}$, are the variable evolutions resulting from the respective simulation.
Per\hh sample \ac{res} and conventional infeed energies are defined in the same manner.
The resulting values for different controller are shown in Table~\ref{tab:closedLoop}. 
The \MinimaxMpcResDroopOnDroopSatOn \ac{mpc} and the prescient \ac{mpc} result in the same Per\hh sample \ac{res} and conventional infeed energies. 
The variation in the per\hh sample cost results from the higher number of switching in prescient \ac{mpc} compared to \MinimaxMpcResDroopOnDroopSatOn \ac{mpc}. 
The considered \ac{mpc} formuation uses a finite horizon and hence, it is not guanteed to deliver the optimal closed-loop performance. 
Furthermore, here the \ac{mpc} formulation can have multiple optimal solutions with same optimal cost which could explain different control decisions.
Also note that -- contrary to open\hh loop prediction -- not all closed\hh loop conventional infeed energies are identical.

%% file: appendix.tex

\subsection{Preliminaries and lemmas used for proving Theorems~\ref{thm:constraint}~and~\ref{thm:cost}}\label{sec:proofof:lemmas}

In the following, we prepare the ground for proving Theorems~\ref{thm:constraint}~and~\ref{thm:cost}, which deal with feasibility and the worst\hh case cost, respectively.
For evaluating feasibility, we will introduce a measure which allows us to decide whether a control is feasible or not.
We want this measure to be always well\hh defined -- also when the control is not feasible.
For this purpose, we reformulate Problem~\ref{prob:minimax} and use a relaxation.

For the proofs, it is convenient to solve the algebraic equations corresponding to the power balance and droop control with saturation for $\rho\okj$.
Therefore, we define a $\rho\okj$\hh candidate explicitly as a function of $u\okj$, $\deltath\okj$, $x\okjm$, $w\okj$, $\jInOneNp$.
The following definitions and Lemma~\ref{lem:rhoxscbsca} are concerned with this $\rho\okj$\hh candidate function and its properties when we assume that $u\okj$, $\deltath\okj$, $x\okjm$, $w\okj$ are independent variables.
Based on the function and its established properties, Lemma~\ref{lem:xscascblemma} and the proofs thereafter analyze the evolution of variables over the prediction horizon $\jInOneNp$.
Please note that we sometimes ease notation by omitting the time indexes $k+j$ and $k+j-1$.

We denote the droop control laws with saturation \eqref{eq:prsok}, \eqref{eq:conventional:saturation:power}, \eqref{eq:pstok} by functions of the independent variables by writing $\prs(\urs,\wrs,\rho)$, $\pth(\uth,\deltath,\rho)$, $\pst(\ust,x,\rho)$, respectively.
Note that one of the important facts about these functions is: for any fixed values of the other variables, they are monotonically increasing in $\rho$.
Let us define auxiliary lower and upper bounds on $\rho$, so that for values of $\rho$ exceeding those bounds it is guaranteed that all units (except those with zero droop gain, the power output of which is not affected by $\rho$) are in lower respectively in upper saturation.
These bounds are determined by the range of possible power to power\hh setpoint differences, \ie,
\begin{subequations}\label{eq:rholohi}\begin{align}
 \rholo &:= \min_{i\in\mathbb{N}_{[1,T+S+R]}\setminus\{i|\chi_i=0\}}\frac{\pimin-\uimax}{\chi_i},\\
 \rhohi &:= \max_{i\in\mathbb{N}_{[1,T+S+R]}\setminus\{i|\chi_i=0\}}\frac{\pimax-\uimin}{\chi_i}.
\end{align}\end{subequations}
Because power output of units is constant for $\rho$ outside of these bounds, we can \wolg add the auxiliary constraint%
\begin{equation}\label{eq:rhorholohi}
 \rholo \leq \rho\okj \leq \rhohi\quad\forall\,\jInOneNp
\end{equation}
to \eqref{eq:minimax:model} in Problem~\ref{prob:minimax}.
This, in turn, enables us to augment \eqref{eq:power:balance} by an additional term.
That is, we define
\begin{multline}\label{eq:ptox}
 \ptox(u,\deltath,x,w,\rho) := \\
 \min\left(0,\rho-\rholo\right) + \max\left(0,\rho-\rhohi\right)+\\
 \mathbf{1}\T\,\pth(\uth,\deltath,\rho) + \mathbf{1}\T\,\pst(\ust,x,\rho) + \mathbf{1}\T\,\prs(\urs,\wrs,\rho) + \mathbf{1}\T\,\wld
\end{multline}
to ensure that, for any combination of $u$, $\deltath$, $x$, $w$, we can always find a $\rho\in\RR$ which satisfies%
\begin{equation}\label{eq:ptoxbalance}
 \ptox(u,\deltath,x,w,\rho) = 0.
\end{equation}

Accordingly, \eqref{eq:ptoxbalance} is identical to \eqref{eq:power:balance} only as long as \eqref{eq:rhorholohi} holds, but the function $\ptox(\cdot,\rho)$ by itself is defined for $\rho\in\RR$ and -- because of the additional term -- is a surjective function of $\rho$.
It is also piecewise affine, continuous and monotonically increasing.
It may be constant on some intervals.
However, it is easy to see that any interval where $\ptox(\cdot,\rho)$ is constant fulfills the following.
1) It is contained in $[\rholo,\,\rhohi]$.
2) All unit power functions are constant on it.
If $\ptox\odo=0$ is located on a constant interval, the lower control layers would settle at a certain $\rho$.
\Wolg, let us assume that always the maximum possible $\rho$ is selected and define a function that provides a $\rho$\hh solution of \eqref{eq:ptoxbalance} as
\begin{multline}\label{eq:rhox}
 \rhox(u,\deltath,x,w) := \max_{\rho\in\RR}\rho\ \ \text{\st}\ \ \ptox(u,\deltath,x,w,\rho) = 0.
\end{multline}

Using this function to define a $\rho\okj$\hh candidate, the constraint \eqref{eq:rhorholohi}, and the aforementioned functions to express the droop control laws with saturation, the robust optimal control Problem~\ref{prob:minimax} can be equivalently stated as
\begin{problem}\label{prob:minimax:reform}
\begin{subequations}\label{eq:minimax:reform}
  \begin{align}
    \min_{\uVec,\,\deltathVec}\ \max_{\wVec}\ J(\pVec, \deltathVec, \deltatk),
  \end{align}
  where $J\odo$ is defined by~\eqref{eq:operationalCost}, subject to the model equations consisting of \eqref{eq:ptox}, \eqref{eq:rhox} and
  \begin{equation}\label{eq:minimax:reform:model}\begin{split}
      \prs\okj &= \prs(\urs\okj,\wrs\okj,\rho\okj),\\
      \pth\okj &= \pth(\uth\okj,\deltath\okj,\rho\okj),\\
      \pst\okj &= \pst(\ust\okj,x\okjm,\rho\okj),\\
      x\okj &= x\okjm - \Ts\,\pst\okj,\\
      \rho\okj &= \rhox(u\okj,\deltath\okj,x\okjm,w\okj),\\
      \rholo &\leq \rho\okj \leq \rhohi,\\
      j &\in \N_{[1, \Npr]},
  \end{split}\end{equation}
  the initial conditions~\eqref{eq:minimax:ic}, control constraints~\eqref{eq:minimax:control} and uncertainty model~\eqref{eq:minimax:uncert}, and the condition that the control $\deltathVec$, $\uVec$ must be feasible \wrt \eqref{eq:ptox}, \eqref{eq:rhox}, \eqref{eq:minimax:reform:model}, \eqref{eq:minimax:ic} $\forall\,\wVec\in[\wminVec,\wmaxVec]$.
\end{subequations}
\end{problem}

\begin{lemma}\label{lem:rhoxscbsca}
$\rhox(u,\deltath,x,w)$ is monotonically decreasing in $x$ and $w$.
That is, consider values of the respective variables corresponding to two scenarios, $x\sca$, $x\scb$ and $w\sca$ and $w\scb$.
The remaining variables are supposed to have the same values in both scenarios. Then,
\begin{subequations}\label{eq:rhoxscbsca}\begin{align}
 &\begin{multlined}[b]\phantom{w\scb \geq w\sca}\makebox[0cm][r]{\ensuremath{x\scb \geq x\sca}}\implies\rhox(u,\deltath,x\scb,w) \leq \rhox(u,\deltath,x\sca,w),\hspace{-10ex}\\ \end{multlined}\label{eq:rhoxscbscaxgeq}\\
 &\begin{multlined}[b]w\scb \geq w\sca\implies\rhox(u,\deltath,x,w\scb) \leq \rhox(u,\deltath,x,w\sca).\hspace{-10ex}\\ \end{multlined}\label{eq:rhoxscbscawgeq}
\end{align}\end{subequations}
\end{lemma}
\begin{proof}
Proving \eqref{eq:rhoxscbscaxgeq} makes use of the fact
\begin{multline}\label{eq:ptoxscbscaxgeq}
 x\scb \geq x\sca\implies\\
 \ptox(u,\deltath,x\scb,w,\rho) \geq \ptox(u,\deltath,x\sca,w,\rho)\ \ \forall\,\rho\in\RR,
\end{multline}
which can be deduced as follows: $\psminn$ and $\psmaxx$ given by \eqref{eq:psiminmaxk} are monotonically increasing in $x$; therefore, noting that $\sat(\pmin,p,\pmax)$ is monotonically increasing in $\pmin$ and $\pmax$, $\pst$ given by \eqref{eq:pstokMain} is monotonically increasing in $x$; therefore, $\ptox$ given by \eqref{eq:ptox} is monotonically increasing in $x$.

According to \eqref{eq:rhox},
\begin{subequations}
 \begin{align}
  \rhox\sca &:=\, \rhox(u,\deltath,x\sca,w) = \max\rho\nonumber\\
  &\quad\text{\st} \ \ \ptox(u,\deltath,x\sca,w,\rho) = 0,\label{eq:rhoxsca}\\
  \rhox\scb &:=\, \rhox(u,\deltath,x\scb,w) = \max\rho\nonumber\\
  &\quad\text{\st} \ \ \ptox(u,\deltath,x\scb,w,\rho) = 0,\label{eq:rhoxscb}
 \end{align}
\end{subequations}
which satisfy
\begin{subequations}\label{eq:ptoxrhox}
 \begin{align}
  \ptox(u,\deltath,x\sca,w,\rhox\sca) &= 0,\label{eq:ptoxrhoxsca}\\
  \ptox(u,\deltath,x\scb,w,\rhox\scb) &= 0.\label{eq:ptoxrhoxscb}
 \end{align}
\end{subequations}
As a consequence of \eqref{eq:ptoxscbscaxgeq} and \eqref{eq:ptoxrhoxscb} we know that
\begin{equation}\label{eq:ptoxxscarhoxscbleqz}
 \ptox(u,\deltath,x\sca,w,\rhox\scb) \leq \ptox(u,\deltath,x\scb,w,\rhox\scb) = 0.
\end{equation}
By definition \eqref{eq:rhoxsca}, $\rhox\sca$ is the maximum $\rho$ for which $\ptox(u,\deltath,x\sca,w,\rho)=0$.
Therefore, if we suppose $\rhox\scb > \rhox\sca$, this would imply
$$\ptox(u,\deltath,x\sca,w,\rhox\scb) \neq 0,$$
and \eqref{eq:ptoxrhoxsca} and the fact that $\ptox(u,\deltath,x\sca,w,\rho)$ is a monotonically increasing function of $\rho$ would imply
$$\ptox(u,\deltath,x\sca,w,\rhox\scb) \geq \ptox(u,\deltath,x\sca,w,\rhox\sca) = 0.$$
Consequently, $\ptox(u,\deltath,x\sca,w,\rhox\scb) > 0$ would be implied, which contradicts \eqref{eq:ptoxxscarhoxscbleqz}.
This concludes the proof of \eqref{eq:rhoxscbscaxgeq}.

Proving \eqref{eq:rhoxscbscawgeq} makes use of the fact
\begin{multline}\label{eq:ptoxscbscawrsgeq}
 w\scb \geq w\sca\implies\\
 \ptox(u,\deltath,x,w\scb,\rho) \geq \ptox(u,\deltath,x,w\sca,\rho)\ \ \forall\,\rho\in\RR,
\end{multline}
which can be deduced as follows: noting that $\sat(\pmin,p,\pmax)$ is monotonically increasing in $\pmin$ and $\pmax$, $\prs$ given by \eqref{eq:prsok} is monotonically increasing in $\wrs$; therefore, $\ptox$ given by \eqref{eq:ptox} is monotonically increasing in $\wrs$; since $\ptox$ is also monotonically increasing in $\wld$, it is monotonically increasing in $w$.
\end{proof}

\begin{lemma}\label{lem:xscascblemma}
Consider Problem~\ref{prob:minimax:reform} and two scenarios with common control actions $u\okj$, $\deltath\okj$ but different disturbance forecasts $w\sca\okj$ and $w\scb\okj$.
For a given $j$, if
\begin{subequations}\begin{align}
 w\scb\okj &\geq w\sca\okj,\label{eq:xscascblemmawscascb}\\
 x\scb\okjm &\geq x\sca\okjm,\label{eq:xscascblemmaxscascb}
\end{align}\end{subequations}
then
\begin{subequations}\begin{align}
 \rho\scb\okj &\leq \rho\sca\okj,\label{eq:xscascblemmarhox}\\
 \pth\scb\okj &\leq \pth\sca\okj,\label{eq:xscascblemmaptkj}\\
 x\scb\okj &\geq x\sca\okj.\label{eq:xscascblemmaxokk}
\end{align}\end{subequations}
\end{lemma}
\begin{proof}
Defining an auxiliary storage power demand $\pstv\okj$ that does not consider \acl{soc}\hh dependent power limits, the storage model \eqref{eq:pstok}, \eqref{eq:storage:dynamics} can be transformed into a model composed of
\begin{subequations}\label{eq:pstvokxokkpstv}\begin{align}
 \pstv\okj &:= \sat\left(\psmin, \ust\okj+\chi_\tst\,\rho\okj, \psmax\right),\label{eq:pstvok}\\
 x\okj &= \sat\left(\xsmin, x\okjm-\Ts\,\pstv\okj, \xsmax\right)\label{eq:xokkpstv}\\
 \shortintertext{and}
 \pskj &= \nicefrac{1}{\Ts}\left(x\okjm-x\okj\right).
\end{align}\end{subequations}
From \eqref{eq:xscascblemmawscascb}, \eqref{eq:xscascblemmaxscascb} and Lemma~\ref{lem:rhoxscbsca} follows \eqref{eq:xscascblemmarhox}.
From \eqref{eq:xscascblemmarhox} and \eqref{eq:conventional:saturation:power} and with $\chi\geq0$ follows \eqref{eq:xscascblemmaptkj}, and from \eqref{eq:xscascblemmarhox} and \eqref{eq:pstvok} follows
\begin{equation}
 \pstv\scb\okj \leq \pstv\sca\okj.
\end{equation}
Using this in \eqref{eq:xokkpstv} directly leads to \eqref{eq:xscascblemmaxokk}.
\end{proof}

Having derived Lemmas~\ref{lem:rhoxscbsca} and~\ref{lem:xscascblemma}, we can now use them to prove Theorems~\ref{thm:constraint} and~\ref{thm:cost}.

\subsection{Proof of Theorem~\ref{thm:constraint}}\label{sec:proofof:thm:constraint}

\begin{proof}
Here, we again refer to the vector respectively matrix definitions \eqref{eq:defineVec} for profiles over the prediction horizon.
We need to show that, for given initial conditions, a control is robustly feasible in the sense of Problem~\ref{prob:minimax} for all possible disturbance realizations $\wVec\in[\wminVec,\wmaxVec]$ \acl{iff} it is feasible for disturbance realizations $\wVec=\wminVec$ and $\wVec=\wmaxVec$.

Consider Problem~\ref{prob:minimax:reform} defined above, which is equivalent to Problem~\ref{prob:minimax}.
Note that for a relaxed problem -- where \eqref{eq:rhorholohi} is abandoned -- any control, $\umin\leq\uVec\leq\umax$, is feasible for any disturbance and initial conditions.
In particular, $\rho\okj$ is always well\hh defined, even if it is not contained in the interval $[\rholo,\rhohi]$.
Therefore, for the relaxed problem, consider the sequence $\rhoVec:=[\rho\okk\,\ldots\,\rho\okNp]$ which is uniquely determined by a given initial state $x'\ok$ and initial switch statuses $\deltath'\ok$, control $\deltathVec':=[\deltath'\okk\,\ldots\,\deltath'\okNp]$, $\uVec':=[u'\okk\,\ldots\,u'\okNp]$ and disturbance forecast $\wVec$.
Whether $\rhoVec$ satisfies $\rholo\leq\rho\okj\leq\rhohi$ $\forall\,\jInOneNp$ indicates if also in the context of Problem~\ref{prob:minimax:reform} the given control is feasible for the given disturbance forecast and initial conditions.

Therefore, we need to show that for any given $x'\ok\in[\xsmin,\xsmax]$, $\deltath'\ok\in\{0,1\}^T$, $\deltathVec'\in\{0,1\}^{T\times\Npr}$, $\uVec'\in[\umin,\umax]$, it holds that%
\begin{align}
 \nonumber\rholo&\leq{\rhoVec\,}_{\left|\begin{smallmatrix*}[l]x\ok=x'\ok,\\ \deltathVec=\deltathVec',\\ \uVec=\uVec',\\ \wVec\end{smallmatrix*}\right.}\leq\rhohi\quad \forall\,\wVec\in[\wminVec,\wmaxVec],\\
 \intertext{\acl{iff}}
 \rholo&\leq{\rhoVec\,}_{\left|\begin{smallmatrix*}[l]x\ok=x'\ok,\\ \deltathVec=\deltathVec',\\ \uVec=\uVec',\\ \wVec\end{smallmatrix*}\right.}\leq\rhohi\quad \begin{aligned}[t]\text{for}\ \wVec&=\wminVec\\ \text{and}\ \wVec&=\wmaxVec.\end{aligned}\label{eq:feasibilityTractable}
\end{align}
Note that the initial switch statuses $\deltath'\ok$ are missing in the condition blocks because they only impact the cost but not the constraints of the problem.

The only if\hh part is trivial as $\wminVec$ and $\wmaxVec$ are already included in the interval $[\wminVec, \wmaxVec]$.
To prove the if\hh part, consider two scenarios with disturbance realizations $\wVec\sca\leq\wVec\scb$.
Both scenarios assume the same control actions over the prediction horizon and identical initial conditions, \ie,%
\begin{subequations}\label{eq:proofcostfacts}\begin{alignat}{2}
 x\sca\ok &= x\scb\ok, \quad & &\!\!\!\!\!\!\deltath\sca\ok = \deltath\scb\ok,\label{eq:proofcostfactsxk}\\
 \deltath\sca\okj &= \deltath\scb\okj\quad & &\forall j\inOneNp,\label{eq:proofcostfactsdeltathkj}\\
 u\sca\okj &= u\scb\okj\quad & &\forall j\inOneNp,\label{eq:proofcostfactsukj}\\
 w\sca\okj &\leq w\scb\okj\quad & &\forall j\inOneNp.\label{eq:proofcostfactswkj}
\end{alignat}\end{subequations}
From iteratively applying Lemma~\ref{lem:xscascblemma} follows%
\begin{equation}\label{eq:proofconstrrho}
 \rho\sca\okj \geq \rho\scb\okj\quad \forall\,\jInOneNp.
\end{equation}
To summarize, if the evolution of the other variables is fixed, we have $\wVec\sca\leq\wVec\scb\implies\rhoVec\sca\geq\rhoVec\scb$.
Now substitute $\wVec\sca=\wminVec$ and $\wVec\scb=\wVec$ and denote the respective $\rhoVec\sca=:\rhoVec_{\left|\wminVec\right.}$ and $\rhoVec\scb=:\rhoVec_{\left|\wVec\right.}$.
Consequently, $\wminVec\leq\wVec\implies\rhoVec_{\left|\wminVec\right.}\geq\rhoVec_{\left|\wVec\right.}$.
Combining this with the result of the substitution $\wVec\sca=\wVec$ and $\wVec\scb=\wmaxVec$ and defining $\rhoVec_{\left|\wmaxVec\right.}$ accordingly gives
\begin{equation}
 \wminVec\leq\wVec\leq\wmaxVec\implies\rhoVec_{\left|\wmaxVec\right.}\leq\rhoVec_{\left|\wVec\right.}\leq\rhoVec_{\left|\wminVec\right.}.
\end{equation}
The condition on the left side of this implication is obviously true, and so we can state that $\rhoVec_{\left|\wVec\right.}$ is bounded by $\rhoVec_{\left|\wmaxVec\right.}$ and $\rhoVec_{\left|\wminVec\right.}$.
If, in turn, each of these bounds is bounded by $\rholo$ and $\rhohi$, then $\rhoVec_{\left|\wVec\right.}$ is bounded by $\rholo$ and $\rhohi$.
This represents the if\hh part of \eqref{eq:feasibilityTractable} and completes the proof.
\end{proof}

\subsection{Proof of Theorem~\ref{thm:cost}}\label{sec:proofof:thm:cost}

\begin{proof}
Using the same matrix respectively vector definitions as before, we need to show that, for any given initial state $x'\ok$ and switch statuses $\deltath'\ok$ as well as control actions $\deltathVec'$ and $\uVec'$, solving the inner maximization problem in \eqref{eq:minimax} reduces to%
\begin{equation}\label{eq:worstcasecostcondition}
 \max_{\wVec\in[\wminVec,\wmaxVec]}
 {J\,}_{\left|\begin{smallmatrix*}[l]x\ok=x'\ok,\\ \deltath\ok=\deltath'\ok,\\ \deltathVec=\deltathVec',\\ \uVec=\uVec',\\ \wVec\end{smallmatrix*}\right.} =
 {J\,}_{\left|\begin{smallmatrix*}[l]x\ok=x'\ok,\\ \deltath\ok=\deltath'\ok,\\ \deltathVec=\deltathVec',\\ \uVec=\uVec',\\ \wVec=\wminVec\end{smallmatrix*}\right.}.
\end{equation}
Consider two scenarios with disturbance realizations $\wVec\sca\leq\wVec\scb$ as in \eqref{eq:proofcostfacts}.
From iteratively applying Lemma~\ref{lem:xscascblemma} follows%
\begin{subequations}\label{eq:proofcostpthx}\begin{alignat}{2}
 \pth\sca\okj &\geq \pth\scb\okj\quad & &\forall j\inOneNp,\\
 x\sca\okj &\leq x\scb\okj\quad & &\forall j\inOneNp.
\end{alignat}\end{subequations}
Note that in particular $x\sca(k+\Npr) \leq x\scb(k+\Npr)$.
With the storage dynamics \eqref{eq:storage:dynamics}, the cost in \eqref{eq:minimax} can be equivalently expressed as
\begin{multline}
 J = - \textstyle\frac{\Cst\T}{\Ts} (x(k+\Npr) - x\ok) + \\
 \textstyle\sum_{j = 1}^{\Npr} \left(\parenTwoLine{\Cth\T \ptkj + \CthOn\T \deltatkj +}
 {\CthSw\T \left|\deltatkj-\deltatkjm\right|}\,\right).
\end{multline}
With \eqref{eq:proofcostfactsdeltathkj}, \eqref{eq:proofcostfactsxk} and \eqref{eq:proofcostpthx}, and recalling that $\Cst$, $\Cth$ are positive, it follows that
\begin{equation}\label{eq:proofcost:Jscascb}
 J\sca \geq J\scb.
\end{equation}
To summarize, if the evolution of the other variables is fixed, we have $\wVec\sca\leq\wVec\scb\implies J\sca \geq J\scb$.
Now substitute $\wVec\sca=\wminVec$ and $\wVec\scb=\wVec$ and denote the respective $J\sca=:J_{\left|\wminVec\right.}$ and $J\scb=:J_{\left|\wVec\right.}$.
Consequently, $\wminVec\leq\wVec\implies J_{\left|\wminVec\right.}\geq J_{\left|\wVec\right.}$.
The condition on the left side of this implication is obviously true, and so we can state that $J_{\left|\wminVec\right.}$ is an upper bound for $J_{\left|\wVec\right.}$.
It is attained by inserting $\wVec=\wminVec$, and so $\wminVec$ maximizes $J_{\left|\wVec\right.}$.
\end{proof}

%% file: paper.bbl
\begin{thebibliography}{10}
\providecommand{\url}[1]{#1}
\csname url@samestyle\endcsname
\providecommand{\newblock}{\relax}
\providecommand{\bibinfo}[2]{#2}
\providecommand{\BIBentrySTDinterwordspacing}{\spaceskip=0pt\relax}
\providecommand{\BIBentryALTinterwordstretchfactor}{4}
\providecommand{\BIBentryALTinterwordspacing}{\spaceskip=\fontdimen2\font plus
\BIBentryALTinterwordstretchfactor\fontdimen3\font minus
  \fontdimen4\font\relax}
\providecommand{\BIBforeignlanguage}[2]{{%
\expandafter\ifx\csname l@#1\endcsname\relax
\typeout{** WARNING: IEEEtran.bst: No hyphenation pattern has been}%
\typeout{** loaded for the language `#1'. Using the pattern for}%
\typeout{** the default language instead.}%
\else
\language=\csname l@#1\endcsname
\fi
#2}}
\providecommand{\BIBdecl}{\relax}
\BIBdecl

\bibitem{REN212018}
{REN21 Secretariat}, ``Renewables 2018 global status report,'' Renewable Energy
  Policy Network for the 21st Century (REN21), c/o UN Environment, 1 rue
  Miollis, Building VII, 75015 Paris, France, Tech. Rep., 2018.

\bibitem{ParLotKhoBah2015}
S.~Parhizi, H.~Lotfi, A.~Khodaei, and S.~Bahramirad, ``State of the art in
  research on microgrids: A review,'' \emph{IEEE Access}, vol.~3, pp. 890--925,
  2015.

\bibitem{PES2018}
{IEEE PES Task Force on Microgrid Stability Analysis}, ``{M}icrogrid
  {S}tability {D}efinitions, {A}nalysis, and {M}odeling (technical report
  {PES}-{TR}66),'' IEEE Power \& Energy Society, Tech. Rep., Apr. 2018.

\bibitem{SchOrtAstRaiSez2014}
\BIBentryALTinterwordspacing
J.~Schiffer, R.~Ortega, A.~Astolfi, J.~Raisch, and T.~Sezi, ``Conditions for
  stability of droop-controlled inverter-based microgrids,'' \emph{Automatica},
  vol.~50, no.~10, pp. 2457--2469, 2014. [Online]. Available:
  \url{http://www.sciencedirect.com/science/article/pii/S0005109814003100}
\BIBentrySTDinterwordspacing

\bibitem{ZiaElbBen2018}
\BIBentryALTinterwordspacing
M.~F. Zia, E.~Elbouchikhi, and M.~Benbouzid, ``Microgrids energy management
  systems: A critical review on methods, solutions, and prospects,''
  \emph{Applied Energy}, vol. 222, pp. 1033--1055, 2018. [Online]. Available:
  \url{http://www.sciencedirect.com/science/article/pii/S0306261918306676}
\BIBentrySTDinterwordspacing

\bibitem{ParRikGli2014}
A.~Parisio, E.~Rikos, and L.~Glielmo, ``A model predictive control approach to
  microgrid operation optimization,'' \emph{IEEE Trans. Control Syst.
  Technol.}, vol.~22, no.~5, pp. 1813--1827, Sep. 2014.

\bibitem{HanNenRaiRei2014}
\BIBentryALTinterwordspacing
C.~A. Hans, V.~Nenchev, J.~Raisch, and C.~Reincke-Collon, ``Minimax model
  predictive operation control of microgrids,'' \emph{IFAC Proceedings
  Volumes}, vol.~47, no.~3, pp. 10\,287--10\,292, 2014, 19th IFAC World
  Congress. [Online]. Available:
  \url{http://www.sciencedirect.com/science/article/pii/S1474667016432465}
\BIBentrySTDinterwordspacing

\bibitem{SalOzkLudWeiHof2018}
\BIBentryALTinterwordspacing
M.~B. Salt{\i}k, L.~{\"{O}}zkan, J.~H.~A. Ludlage, S.~Weiland, and P.~M.~J.
  Van~den Hof, ``An outlook on robust model predictive control algorithms:
  Reflections on performance and computational aspects,'' \emph{Journal of
  Process Control}, vol.~61, pp. 77--102, 2018. [Online]. Available:
  \url{http://www.sciencedirect.com/science/article/pii/S0959152417301968}
\BIBentrySTDinterwordspacing

\bibitem{Lof2003}
J.~L{\"{o}}fberg, ``Minimax approaches to robust model predictive control,''
  Ph.D. dissertation, Department of Electrical Engineering, Link{\"{o}}ping
  University, Link{\"{o}}ping, Sweden, 2003.

\bibitem{BemBorMor2003}
A.~Bemporad, F.~Borrelli, and M.~Morari, ``Min-max control of constrained
  uncertain discrete-time linear systems,'' \emph{IEEE Transactions on
  Automatic Control}, vol.~48, no.~9, pp. 1600--1606, Sept 2003.

\bibitem{NutLohWanBla2015}
I.~U. Nutkani, P.~C. Loh, P.~Wang, and F.~Blaabjerg, ``Cost-prioritized droop
  schemes for autonomous ac microgrids,'' \emph{IEEE Transactions on Power
  Electronics}, vol.~30, no.~2, pp. 1109--1119, Feb 2015.

\bibitem{NutLohWanBla2016}
------, ``Linear decentralized power sharing schemes for economic operation of
  ac microgrids,'' \emph{IEEE Transactions on Industrial Electronics}, vol.~63,
  no.~1, pp. 225--234, Jan 2016.

\bibitem{CheCheLiMenZheGueAbb2017}
F.~Chen, M.~Chen, Q.~Li, K.~Meng, Y.~Zheng, J.~M. Guerrero, and D.~Abbott,
  ``Cost-based droop schemes for economic dispatch in islanded microgrids,''
  \emph{IEEE Transactions on Smart Grid}, vol.~8, no.~1, pp. 63--74, Jan 2017.

\bibitem{DorSimBul2016}
F.~D{\"{o}}rfler, J.~W. Simpson-Porco, and F.~Bullo, ``Breaking the hierarchy:
  Distributed control and economic optimality in microgrids,'' \emph{IEEE
  Transactions on Control of Network Systems}, vol.~3, no.~3, pp. 241--253,
  Sept 2016.

\bibitem{CaoLin2005}
Y.-Y. Cao and Z.~Lin, ``Min-max mpc algorithm for lpv systems subject to input
  saturation,'' \emph{IEE Proceedings - Control Theory and Applications}, vol.
  152, no.~3, pp. 266--272, May 2005.

\bibitem{HuaLiLinXi2011}
\BIBentryALTinterwordspacing
H.~Huang, D.~Li, Z.~Lin, and Y.~Xi, ``An improved robust model predictive
  control design in the presence of actuator saturation,'' \emph{Automatica},
  vol.~47, no.~4, pp. 861--864, 2011. [Online]. Available:
  \url{http://www.sciencedirect.com/science/article/pii/S0005109811000604}
\BIBentrySTDinterwordspacing

\bibitem{OraBak2015}
\BIBentryALTinterwordspacing
J.~Oravec and M.~Bako{\v s}ov{\'{a}}, ``Robust model predictive control based
  on nominal system optimization and control input saturation,''
  \emph{IFAC-PapersOnLine}, vol.~48, no.~14, pp. 314--319, 2015, 8th IFAC
  Symposium on Robust Control Design ROCOND 2015. [Online]. Available:
  \url{http://www.sciencedirect.com/science/article/pii/S2405896315015955}
\BIBentrySTDinterwordspacing

\bibitem{VanKooMeeVan2013}
T.~L. Vandoorn, J.~D. Kooning, B.~Meersman, and L.~Vandevelde, ``Voltage-based
  droop control of renewables to avoid on–off oscillations caused by
  overvoltages,'' \emph{IEEE Transactions on Power Delivery}, vol.~28, no.~2,
  pp. 845--854, April 2013.

\bibitem{DasAltHanSorFlyAbi2018}
K.~Das, M.~Altin, A.~D. Hansen, P.~E. S{\o}rensen, D.~Flynn, and H.~Abildgaard,
  ``Wind power support during overfrequency emergency events,'' \emph{CIGRE
  Science \& Engineering}, vol.~9, pp. 73--83, 2018.

\bibitem{NeeJohDelGonLav02016}
J.~{Neely}, J.~{Johnson}, J.~{Delhotal}, S.~{Gonzalez}, and M.~{Lave},
  ``Evaluation of pv frequency-watt function for fast frequency reserves,'' in
  \emph{2016 IEEE Applied Power Electronics Conference and Exposition (APEC)},
  March 2016, pp. 1926--1933.

\bibitem{XinLiuWanGanYan2013}
H.~{Xin}, Y.~{Liu}, Z.~{Wang}, D.~{Gan}, and T.~{Yang}, ``A new frequency
  regulation strategy for photovoltaic systems without energy storage,''
  \emph{IEEE Transactions on Sustainable Energy}, vol.~4, no.~4, pp. 985--993,
  Oct 2013.

\bibitem{PauKinFaiGha2016}
J.~Pauphilet, D.~Kiner, D.~Faille, and L.~El~Ghaoui, ``A tractable numerical
  strategy for robust milp and application to energy management,'' in
  \emph{2016 IEEE 55th Conference on Decision and Control (CDC)}, Dec 2016, pp.
  1490--1495.

\bibitem{HanSopRaiReiPat2020}
\BIBentryALTinterwordspacing
C.~A. Hans, P.~Sopasakis, J.~Raisch, C.~Reincke-Collon, and P.~Patrinos,
  ``Risk-averse model predictive operation control of islanded microgrids,''
  \emph{IEEE Transactions on Control Systems Technology}, vol.~28, no.~6, pp.
  2136--2151, Nov. 2020. [Online]. Available:
  \url{https://ieeexplore.ieee.org/document/8792379}
\BIBentrySTDinterwordspacing

\bibitem{KriHanSchRaiKra2017}
A.~Krishna, C.~A. Hans, J.~Schiffer, J.~Raisch, and T.~Kral, ``Steady state
  evaluation of distributed secondary frequency control strategies for
  microgrids in the presence of clock drifts,'' in \emph{2017 25th
  Mediterranean Conference on Control and Automation (MED)}, July 2017, pp.
  508--515.

\bibitem{SchHanKraOrtRai2017}
J.~Schiffer, C.~A. Hans, T.~Kral, R.~Ortega, and J.~Raisch, ``Modeling,
  analysis, and experimental validation of clock drift effects in low-inertia
  power systems,'' \emph{IEEE Transactions on Industrial Electronics}, vol.~64,
  no.~7, pp. 5942--5951, July 2017.

\bibitem{HynAth2018}
R.~J. Hyndman and G.~Athanasopoulos, \emph{Forecasting: principles and
  practice}.\hskip 1em plus 0.5em minus 0.4em\relax OTexts, 2018.

\bibitem{BemMor1999}
\BIBentryALTinterwordspacing
A.~Bemporad and M.~Morari, ``Control of systems integrating logic, dynamics,
  and constraints,'' \emph{Automatica}, vol.~35, no.~3, pp. 407--427, 1999.
  [Online]. Available:
  \url{http://www.sciencedirect.com/science/article/pii/S0005109898001782}
\BIBentrySTDinterwordspacing

\bibitem{Lof2004}
J.~L{\"{o}}fberg, ``{YALMIP} : a toolbox for modeling and optimization in
  {MATLAB},'' in \emph{2004 IEEE International Symposium on Computer Aided
  Control Systems Design}, Sept 2004, pp. 284--289.

\bibitem{gurobi}
\BIBentryALTinterwordspacing
{Gurobi Optimization LLC}, ``Gurobi optimizer reference manual,'' 2018.
  [Online]. Available: \url{https://www.gurobi.com}
\BIBentrySTDinterwordspacing

\bibitem{ARM2011}
ARM, ``{A}tmospheric {R}adiation {M}easurement {C}limate {R}esearch {F}acility,
  {S}urface {M}eteorology {S}ystem ({MET}), {E}astern {N}orth {A}tlantic
  {F}acility {ARM} {D}ata {A}rchive: {O}ak {R}idge, {USA},'' accessed July 14,
  2011.

\bibitem{KlaHan2018}
\BIBentryALTinterwordspacing
C.~A. Hans and E.~Klages, ``Very short term time-series forecasting of solar
  irradiance without exogenous inputs,'' in \emph{6th International Conference
  on Time Series and Forecasting (ITISE)}, 2019. [Online]. Available:
  \url{https://arxiv.org/abs/1810.07066}
\BIBentrySTDinterwordspacing

\end{thebibliography}
